\documentclass[a4,11pt]{article}
\usepackage[dvips]{graphicx}
\usepackage{amssymb}
\usepackage{amstext}
\usepackage{amsmath}
\usepackage{amscd}
\usepackage{amsthm}
\usepackage{amsfonts}
\usepackage{enumerate}
\usepackage{graphicx}
\usepackage{latexsym}
\usepackage{mathrsfs}
\usepackage{mathtools}
\usepackage{cases}
\usepackage[svgnames]{xcolor}
\usepackage{verbatim}
\usepackage{bm}
\usepackage{tikz}
\usetikzlibrary{arrows.meta}
\newtheorem{thm}{Theorem}[section]

\newtheorem{lem}[thm]{Lemma}

\newtheorem{ass}[thm]{Assumption}
\theoremstyle{definition}
\newtheorem{rem}[thm]{Remark}
\newtheorem{dfn}[thm]{Definition}

\newtheorem*{claim*}{Claim}

\numberwithin{equation}{section}

\title{\vspace{-3cm}\textbf{Scattering by the local perturbation of an open periodic waveguide in the half plane}}

\author{Takashi FURUYA}

\date{}
\begin{document}
\maketitle
\begin{abstract}
We consider the scattering problem of the local perturbation of an open periodic waveguide in the half plane. Recently in \cite{Kirsch and Lechleiter2}, a new radiation condition was introduced in order to solve the unperturbed case. In this paper, under the same radiation condition with \cite{Kirsch and Lechleiter2} (Definition 2.4) and an additional assumption (Assumption 1.1) we show the well-posedness of the perturbed scattering problem.
\end{abstract}
\section{Introduction}
Let $k>0$ be the wave number, and let $\mathbb{R}^2_{+}:=\mathbb{R}\times (0, \infty)$ be the upper half plane, and let $W:=\mathbb{R}\times (0, h)$ be the waveguide in $\mathbb{R}^2_{+}$. We denote by $\Gamma_a:=\mathbb{R}\times\{ a\}$ for $a>0$. Let $n \in L^{\infty}(\mathbb{R}^2_{+})$ be real value, $2\pi$-periodic with respect to $x_1$ (that is, $n(x_1+2\pi,x_2)=n(x_1,x_2 )$ for all $x=(x_1,x_2) \in \mathbb{R}^2_{+}$), and equal to one for $x_2>h$. We assume that there exists a constant $n_0>0$ such that $n \geq n_0$ in $\mathbb{R}^2_{+}$. Let $q \in L^{\infty}(\mathbb{R}^2_{+})$ be real value with the compact support in $W$. We denote by $Q:=\mathrm{supp}q$. In this paper, we consider the following scattering problem: For fixed $y \in \mathbb{R}^2_{+} \setminus \overline{W}$, determine the scattered field $u^{s} \in H^{1}_{loc}(\mathbb{R}^2_{+})$ such that
\begin{equation}
\Delta u^{s}+k^2(1+q)nu^{s}=-k^{2}qnu^{i}(\cdot, y) \ \mathrm{in} \ \mathbb{R}^2_{+}, \label{1.1}
\end{equation}
\begin{equation}
u^{s}=0 \ \mathrm{on} \ \Gamma_0, \label{1.2}
\end{equation}
Here, the incident field $u^{i}$ is given by $u^{i}(x,y)=G_n(x,y)$, where $G_n$ is the Dirichlet Green's function in the upper half plane $\mathbb{R}^2_{+}$ for $\Delta +k^2n$, that is,
\begin{equation}
G_n(x,y):=G(x,y)+\tilde{u}^{s}(x,y), \label{1.3}
\end{equation}
where $G(x,y):=\Phi_k(x,y)-\Phi_k(x,y^{*})$ is the Dirichlet Green's function in $\mathbb{R}^2_+$ for $\Delta +k^2$, and $y^{*}=(y_1, -y_2)$ is the reflected point of $y$ at $\mathbb{R}\times \{0\}$. Here, $\Phi_k(x,y)$ is the fundamental solution to Helmholtz equation in $\mathbb{R}^2$, that is, 
\begin{equation}
\Phi_k(x,y):= \displaystyle \frac{i}{4}H^{(1)}_0(k|x-y|), \ x \neq y. \label{1.4}
\end{equation}
$\tilde{u}^{s}$ is the scattered field of the unperturbed problem by the incident field $G(x,y)$, that is, $\tilde{u}^{s}$ vanishes for $x_2=0$ and solves
\begin{equation}
\Delta \tilde{u}^{s}+k^2n\tilde{u}^{s}=k^{2}(1-n)G(\cdot, y) \ \mathrm{in} \ \mathbb{R}^2_{+}. \label{1.5}
\end{equation}
If we impose a suitable radiation condition introduced by Kirsch and Lechleiter \cite{Kirsch and Lechleiter2}, the unperturbed solution $\tilde{u}^{s}$ is uniquely determined. Later, we will explain the exact definition of this radiation condition (see Definition 2.4). 
\par
In order to show the well-posedness of the perturbed scattering problem (\ref{1.1})--(\ref{1.2}), we make the following assumption.
\begin{ass}
We assume that $k^2$ is not the point spectrum of $\frac{1}{(1+q)n}\Delta$ in $H^{1}_{0}(\mathbb{R}^{2}_{+})$, that is, evey $v \in H^{1}(\mathbb{R}^{2}_{+})$ which satisfies
\begin{equation}
\Delta v+k^2(1+q)nv=0 \ \mathrm{in} \ \mathbb{R}^2_{+}, \label{1.6}
\end{equation}
\begin{equation}
v=0 \ \mathrm{on} \ \Gamma_0, \label{1.7}
\end{equation}
has to vanish for $x_2>0$.
\end{ass}
If we assume that $q$ and $n$ satisfy in addition that $\partial_2 \bigl((1+q)n\bigr) \geq 0$ in $W$, then $v$ which satisfies (\ref{1.6})--(\ref{1.7}) vanishes, that is, under this assumption all of $k^2$ is not the point spectrum of $\frac{1}{(1+q)n}\Delta$. We will prove it in Section 6. Our aim in this paper is to show the following theorem.
\begin{thm}
Let Assumptions 1.1 and 2.1 hold and let $k>0$ be regular in the sense of Definition 2.3 and let $f \in L^{2}(\mathbb{R}^2_{+})$ such that $\mathrm{supp}f=Q$. Then, there exists a unique solution $u \in H^{1}_{loc}(\mathbb{R}^2_{+})$ such that
\begin{equation}
\Delta u+k^2(1+q)nu=f \ \mathrm{in} \ \mathbb{R}^2_{+}, \label{1.8}
\end{equation}
\begin{equation}
u=0 \ \mathrm{on} \ \Gamma_0, \label{1.9}
\end{equation}
and $u$ satisfies the radiation condition in the sense of Definition 2.4.
\end{thm}
Roughly speaking, the radiation condition of Definition 2.4  requires that we have a decomposition of the solution $u$ into $u^{(1)}$ which decays in the direction of $x_1$, and a finite combination $u^{(2)}$ of {\it propagative modes} which does not decay, but it exponentially decays in the direction of $x_2$.
\par
This paper is organized as follows. In Section 2, we briefly recall a radiation condition introduced in \cite{Kirsch and Lechleiter2}, and show that the solution of (\ref{2.1})--(\ref{2.2}) has an integral representation (\ref{2.18}). Under the radiation condition in the sense of Definition 2.4, we show the uniqueness of $u^{(2)}$ and $u^{(1)}$ in Section 3 and 4, respectively. In Section 5, we show the existence of $u$. In Section 6, we will give an example of $n$ and $q$ with respect to Assumption 1.1.
\section{A radiation condition}
In Section 2, we briefly recall a radiation condition introduced in \cite{Kirsch and Lechleiter2}. Let $f \in L^{2}(\mathbb{R}^2_{+})$ have the compact support in $W$.  First, we consider the following problem: Find $u \in H^{1}_{loc}(\mathbb{R}^2_{+})$ such that
\begin{equation}
\Delta u+k^2nu=f \ \mathrm{in} \ \mathbb{R}^2_{+}, \label{2.1}
\end{equation}
\begin{equation}
u=0 \ \mathrm{on} \ \Gamma_0. \label{2.2}
\end{equation}
(\ref{2.1}) is understood in the variational sense, that is,
\begin{equation}
\int_{\mathbb{R}^2_{+}} \bigl[ \nabla u \cdot \nabla \overline{\varphi}-k^2nu\overline{\varphi} \bigr]dx=-\int_W f \overline{\varphi}dx, \label{2.3}
\end{equation}
for all $\varphi \in H^{1}(\mathbb{R}^2_{+})$, with compact support. In such a problem, it is natural to impose the {\it upward propagating radiation condition}, that is, $u(\cdot, h) \in L^{\infty}(\mathbb{R})$ and 
\begin{equation}
u(x)=2\int_{\Gamma_h}u(y)\frac{\partial\Phi_k(x,y)}{\partial y_2} ds(y)=0,\ x_2>h. \label{2.4}
\end{equation}
However, even with this condition we can not expect the uniqueness of this problem. (see Example 2.3 of \cite{Kirsch and Lechleiter2}.) In order to introduce a {\it suitable radiation condition}, Kirsch and Lechleiter discussed limiting absorption solution of this problem, that is, the limit of the solution $u_{\epsilon}$ of $\Delta u_{\epsilon}+(k+i\epsilon)^2nu_{\epsilon}=f$ as $\epsilon \to 0$. For the details, we refer to \cite{Kirsch and Lechleiter1, Kirsch and Lechleiter2}.
\par
Let us prepare for the exact definition of the radiation condition. First we recall that the {\it Floquet Bloch transform} $T_{per} : L^{2}(\mathbb{R}) \to L^{2}\bigl( (0, 2\pi) \times (-1/2, 1/2) \bigr)$ is defined by 
\begin{equation}
T_{per}f(t, \alpha) = \tilde{f}_{\alpha}(t) := \sum_{m \in \mathbb{Z}}f(t+ 2\pi m)e^{-i\alpha(t+ 2\pi m)}, \label{2.5}
\end{equation}
for $(t, \alpha) \in (0, 2\pi) \times (-1/2, 1/2)$. The inverse transform is given by
\begin{equation}
T^{-1}_{per}g(t) = \int^{1/2}_{-1/2}g(t, \alpha)e^{i\alpha t}d\alpha,\ t \in \mathbb{R}. \label{2.6}
\end{equation}
By taking the Floquet Bloch transform with respect to $x_1$ in (\ref{2.1})--(\ref{2.2}), we have for $\alpha \in [-1/2, 1/2]$
\begin{equation}
\Delta \tilde{u}_{\alpha}+2i\alpha \frac{\partial \tilde{u}_{\alpha}}{\partial x_1} + (k^2n-\alpha^2)\tilde{u}_\alpha=\tilde{f}_{\alpha} \ \mathrm{in} \ (0,2\pi) \times (0, \infty). \label{2.7}
\end{equation}
\begin{equation}
\tilde{u}_\alpha=0 \ \mathrm{on} \ (0,2\pi)\times \{0 \}. \label{2.8}
\end{equation}
By taking the Floquet Bloch transform with respect to $x_1$ in (\ref{2.4}), $\tilde{u}_\alpha$ satisfies the {\it Rayleigh expansion} of the form 
\begin{equation}
\tilde{u}_{\alpha}(x)=\sum_{n \in \mathbb{Z}}u_{n}(\alpha)e^{inx_1+i\sqrt{k^2-(n+\alpha)^2}(x_2-h)}, \  x_2>h, \label{2.9}
\end{equation}
where $u_{n}(\alpha):=(2\pi)^{-1}\int_{0}^{2\pi}u_{\alpha}(x_1,h)e^{-inx_1}dx_1$ are the Fourier coefficients of $u_{\alpha}(\cdot,h)$, and $\sqrt{k^2-(n+\alpha)^2}=i\sqrt{(n+\alpha)^2-k^2}$ if $n+\alpha>k$. 
\par
We denote by $C_{R}:=(0,2\pi) \times (0, R)$ for $R \in (0,\infty]$, and  $H^{1}_{per}(C_R)$ the subspace of the $2 \pi$-periodic function in $H^{1}(C_R)$. We also denote by $H^{1}_{0,per}(C_{R}):=\{u \in H^{1}_{per}(C_{R}) : u = 0 \ \mathrm{on}\ (0,2\pi)\times \{0 \} \}$ that is equipped with $H^{1}(C_R)$ norm. The space $H^{1}_{0,per}(C_{R})$ has the inner product of the form
\begin{equation}
\langle u, v \rangle_{*}=\int_{C_h}\nabla u \cdot \nabla \overline{v}dx + 2\pi \sum_{n \in \mathbb{Z}}\sqrt{n^2+1}u_n\overline{v_n}, \label{2.10}
\end{equation}
where $u_n=(2\pi)^{-1}\int_{0}^{2\pi}u(x_1,R)e^{-inx_1}dx_1$.
The problem (\ref{2.7})--(\ref{2.9}) is equivalent to the following operator equation (see section 3 in \cite{Kirsch and Lechleiter2}),
\begin{equation}
\tilde{u}_{\alpha}-K_{\alpha}\tilde{u}_{\alpha}=\tilde{f}_{\alpha} \ \mathrm{in} \ H^{1}_{0,per}(C_h),\label{2.11}
\end{equation}
where the operator $K_{\alpha}:H^{1}_{0,per}(C_h) \to H^{1}_{0,per}(C_h)$ is defined by
\begin{eqnarray}
\langle K_{\alpha}u, v \rangle_{*}&=&-\int_{C_h}\left[i\alpha \biggl(u \frac{\partial \overline{v}}{\partial x_1} -\overline{v}\frac{\partial \overline{u}}{\partial x_1}
\biggr)+(\alpha^2-k^2n)u\overline{v}\right]dx 
\nonumber\\
&+& 2\pi i \sum_{|n+\alpha|\leq k}u_n\overline{v_n} \bigl( \sqrt{k^2-(n+\alpha)^2}-i\sqrt{n^2+1} \bigr)
\nonumber\\
&+& 2\pi \sum_{|n+\alpha|> k}u_n\overline{v_n} \bigl(\sqrt{n^2+1}- \sqrt{(n+\alpha)^2-k^2} \bigr).
 \label{2.12}
\end{eqnarray}
For several $\alpha \in (-1/2, 1/2]$, the uniqueness of this problem fails. We call $\alpha$ {\it exceptional values} if the operator $I-K_{\alpha}$ fails to be injective. For the difficulty of treatment of $\alpha$ such that $|\alpha+l|=k$ for some $l \in \mathbb{Z}$ in periodic scattering problem, we set $A_k:=\{\alpha \in (-1/2, 1/2]: \exists l \in \mathbb{Z} \ s.t. \ |\alpha+l|=k \}$, and make the following assumption:
\begin{ass}
For every $\alpha \in A_k$, $I-K_{\alpha}$ has to be injective.
\end{ass}
The following properties of exceptional values was shown in \cite{Kirsch and Lechleiter2}.

\begin{lem}
Let Assumption 2.1 hold. Then, there exists only finitely many exceptional values $\alpha \in (-1/2, 1/2]$. Furthermore, if $\alpha$ is an exceptional value, then so is $-\alpha$. Therefore, the set of exceptional values can be described by $\{\alpha_j:j\in J \}$ where some $J \subset \mathbb{Z}$ is finite and $\alpha_{-j}=-\alpha_j$ for $j \in J$. For each exceptional value $\alpha_j$ we define 
\begin{equation}
X_j:=\left\{ \phi \in H^{1}_{loc}(\mathbb{R}^2_+):\begin{array}{cc}
      \Delta \phi+2i\alpha_j\frac{\partial \phi}{\partial x_1}+(k^2n-\alpha^2)\phi=0 \ \mathrm{in} \ \mathbb{R}^2_+, \\
      \phi=0 \ \mathrm{for} \ x_2=0, \ \ \ \phi \ \mathrm{is} \ 2\pi \mathrm{-periodic} \ \mathrm{for}\ x_1, \\
      \phi \ \mathrm{satisfies \ the \ Rayleigh\ expansion}\ (\ref{2.9})
    \end{array}
\right\} \nonumber
\end{equation}
Then, $X_j$ are finite dimensional. We set $m_j=\mathrm{dim}X_j$. Furthermore, $\phi \in X_j$ is evanescent, that is, there exists $c>0$ and $\delta>0$ such that $|\phi(x)|, \ |\nabla \phi(x)|\leq ce^{-\delta |x_2|}$ for all $x\in \mathbb{R}^2_+$.
\end{lem}
Next, we consider the following eigenvalue problem in $X_j$: Determine $d \in \mathbb{R}$ and $\phi \in X_j$ such that
\begin{equation}
\int_{C_{\infty}}\left[-i\frac{\partial \phi}{\partial x_1}+\alpha_j \phi \right] \overline{\psi} dx= dk\int_{C_{\infty}}n\phi \overline{\psi}dx,\label{2.13}
\end{equation}
for all $\psi \in X_j$. We denote by the eigenvalues $d_{l,j}$ and eigenfunction $\phi_{l,j}$ of this problem, that is,
\begin{equation}
\int_{C_{\infty}}\left[-i\frac{\partial \phi_{l,j}}{\partial x_1}+\alpha_j \phi_{l,j} \right] \overline{\psi} dx= d_{l,j}k\int_{C_{\infty}}n\phi_{l,j} \overline{\psi}dx,\label{2.14}
\end{equation}
for every $l=1,...,m_j$ and $j \in J$. We normalize the eigenfunction $\{\phi_{l,j}: l=1,...,m_j \}$ such that
\begin{equation}
k\int_{C_{\infty}}n\phi_{l,j} \overline{\phi_{l',j}}dx=\delta_{l,l'},\label{2.15}
\end{equation}
for all $l, l'$. We will assume that the wave number $k>0$ is {\it regular} in the following sense.
\begin{dfn}
$k>0$ is {\it regular} if $d_{l,j}\neq 0$ for all $l=1,...m_j$ and $j \in J$.
\end{dfn}
Now we are ready to define the radiation condition. 
\begin{dfn}
Let Assumptions 2.1 hold, and let $k>0$ be regular in the sense of Definition 2.3. We set
\begin{equation}
\psi^{\pm}(x_1):=\frac{1}{2} \left[ 1\pm \frac{2}{\pi}\int_{0}^{x_1/2}\frac{sint}{t}dt \right] , \ x_1 \in \mathbb{R}.\label{2.16}
\end{equation}
Then, $u \in H^{1}_{loc}(\mathbb{R}^2_{+})$ satisfies the {\it radiation condition} if $u$ satisfies the upward propagating radiation condition (\ref{2.4}), and has a decomposition in the form $u=u^{(1)}+u^{(2)}$ where $u^{(1)} \bigl|_{\mathbb{R} \times (0,R)} \in H^{1}(\mathbb{R} \times (0,R))$ for all $R>0$, and $u^{(2)}\in L^{\infty}(\mathbb{R}^{2}_{+})$ has the following form
\begin{equation}
u^{(2)}(x)=\psi^{+}(x_1)\sum_{j \in J} \sum_{d_{l,j}>0}a_{l,j}\phi_{l,j}(x)+\psi^{-}(x_1)\sum_{j \in J} \sum_{d_{l,j}<0}a_{l,j}\phi_{l,j}(x) \label{2.17}
\end{equation}
where some $a_{l,j} \in \mathbb{C}$, and $\{d_{l,j},\phi_{l,j}: l=1,...,m_j \}$ are normalized eigenvalues and eigenfunctions of the problem (\ref{2.8}). 
\end{dfn}

\begin{rem}
It is obvious that we can replace $\psi^{+}$ by any smooth functions $\tilde{\psi}^{\pm}$ with $\tilde{\psi}^{+}(x_1)=1+\mathcal{O}(1/x_1)$ as $x_1\to \infty$ and $\tilde{\psi}^{+}(x_1)=\mathcal{O}(1/x_1)$ as $x_1\to -\infty$ and $\frac{d}{dx_1}\tilde{\psi}^{+}(x_1)\to 0$ as $|x_1|\to \infty$ (and analogously for $\psi^{-}$).
\end{rem}

The following was shown in Theorems 2.2, 6.6, and 6.8 of \cite{Kirsch and Lechleiter2}.
\begin{thm}
Let Assumptions 2.1 hold and let $k>0$ be regular in the sense of Definition 2.3. For every $f \in L^{2}(\mathbb{R}^2_{+})$ with the compact support in $W$, there exists a unique solution $u_{k+i \epsilon} \in H^{1}(\mathbb{R}^{2}_{+})$ of the problem (\ref{2.1})--(\ref{2.2}) replacing $k$ by $k+i\epsilon$. Furthermore, $u_{k+i \epsilon}$ converge as $\epsilon \to +0$ in $H^{1}_{loc}(\mathbb{R}^{2}_{+})$ to some $u \in H^{1}_{loc}(\mathbb{R}^{2}_{+})$ which satisfy (\ref{2.1})--(\ref{2.2}) and the radiation condition in the sense of Definition 2.4. Furthermore, the solution $u$ of this problem is uniquely determined.
\end{thm}
We have recalled the radiation condition and its properties. Finally in this section, we will show the following integral representation.
\begin{lem}
Let $f \in L^2(\mathbb{R}^2_+)$ have a compact support in $W$, and let $u$ be a solution of (\ref{2.1})--(\ref{2.2}) which satisfying the radiation condition in the sense of Definition 2.4. Then, $u$ has an integral representation of the form
\begin{equation}
u(x)=k^2\int_{W} (n(y)-1) u(y)G(x,y)dy-\int_{W} f(y)G(x,y)dy, \ \ x \in \mathbb{R}^2_+ \label{2.18} 
\end{equation}
\end{lem}
\begin{proof}[Proof of Lemma 2.7]
Let $\epsilon >0$ be small enough and let $u_{\epsilon} \in H^{1}(\mathbb{R}^{2}_{+})$ be a solution of the problem (\ref{2.1})--(\ref{2.2}) replacing $k$ by $k+i\epsilon$, that is, $u_{\epsilon}$ satisfies
\begin{equation}
\Delta u_{\epsilon}+(k+i\epsilon)^2nu_{\epsilon}=f \ \mathrm{in} \ \mathbb{R}^2_{+}, \label{2.19}
\end{equation}
\begin{equation}
u_{\epsilon}=0 \ \mathrm{on} \ \Gamma_0. \label{2.20}
\end{equation}
Let $G_{\epsilon}(x,y)$ be the Dirichlet Green's function in the upper half plane $\mathbb{R}^2_{+}$ for $\Delta +(k+i\epsilon)^2$. Let $x \in \mathbb{R}^{2}_+$ be always fixed such that $x_2>R$. Let $r>0$ be large enough such that $x \in B_r(0)$ where $B_r(0) \subset \mathbb{R}^{2}$ be a open ball with center $0$ and radius $r>0$. By Green's representation theorem in $B_r(0)\cap \mathbb{R}^{2}_+$ we have
\begin{eqnarray}
u_{\epsilon}(x)&=&\int_{\partial B_r(0) \cap \mathbb{R}^{2}_+}\bigl[\frac{\partial u_{\epsilon}}{\partial \nu}(y)G_{\epsilon}(x,y)-u_{\epsilon}(y)\frac{\partial G_{\epsilon}}{\partial \nu}(x,y)\bigr]ds(y)
\nonumber\\
&-&\int_{B_r(0)\cap \mathbb{R}^{2}_+} \bigl[\Delta u_{\epsilon}(y)+(k+i\epsilon)^2u_{\epsilon}(y)\bigr]G_{\epsilon}(x,y)dy
\nonumber\\
&=&\int_{\partial B_r(0) \cap \mathbb{R}^{2}_+}\bigl[\frac{\partial u_{\epsilon}}{\partial \nu}(y)G_{\epsilon}(x,y)-u_{\epsilon}(y)\frac{\partial G_{\epsilon}}{\partial \nu}(x,y)\bigr]ds(y)
\nonumber\\
&+&(k+i\epsilon)^2\int_{B_r(0)\cap \mathbb{R}^{2}_+} (n(y)-1) u_{\epsilon}(y)G_{\epsilon}(x,y)dy
\nonumber\\
&-&
\int_{B_r(0)\cap \mathbb{R}^{2}_+} f(y)G_{\epsilon}(x,y)dy.
\label{2.21} 
\end{eqnarray}
Since $u_{\epsilon} \in H^{1}(\mathbb{R}^{2}_{+})$, the first term of the right hand side converges to zero as $r \to \infty$. Therefore, as $r \to \infty$ we have for $x \in \mathbb{R}^{2}_+$
\begin{equation}
u_{\epsilon}(x)=(k+i\epsilon)^2\int_{W} (n(y)-1) u_{\epsilon}(y)G_{\epsilon}(x,y)dy-\int_{W} f(y)G_{\epsilon}(x,y)dy.\label{2.22} 
\end{equation}
We will show that (\ref{2.22}) converges as $\epsilon \to 0$ to
\begin{equation}
u(x)=k^2\int_{W} (n(y)-1) u(y)G(x,y)dy-\int_{W} f(y)G(x,y)dy.\label{2.23} 
\end{equation}
Indeed, by the argument in (3.8) and (3.9) of \cite{Chandler and Christopher}, $G_{\epsilon}(x,y)$ is of the estimation
\begin{equation}
|G_{\epsilon}(x,y)| \leq C \frac{x_2 y_2}{1+|x-y|^{3/2}}, \ |x-y|>1, \label{2.24} 
\end{equation}
where above $C$ is independent of $\epsilon>0$. Then, by Lebesgue dominated convergence theorem we have the second integral in (\ref{2.22}) converges as $\epsilon \to 0$ to one in (\ref{2.23}). So, we will consider the convergence of the first integral in (\ref{2.22}). 
\par
By the beginning of the proof of Theorem 6.6 in \cite{Kirsch and Lechleiter2}, $u_{\epsilon}$ can be of the form $u_{\epsilon}=u^{(1)}_{\epsilon}+u^{(2)}_{\epsilon}$ where $u^{(1)}_{\epsilon}$ converges to $u^{(1)}$ in $H^{1}(W)$, and $u^{(2)}_{\epsilon}$ is of the form for $x \in W$
\begin{equation}
u^{(2)}_{\epsilon}(x)=\sum_{j \in J} \sum_{l=1}^{m_j}y_{l,j}\int ^{1/2}_{-1/2}\frac{e^{i\alpha x_1}}{i\epsilon-d_{l,j}\alpha}d\alpha \ \phi_{l,j}(x), \label{2.25} 
\end{equation}
which converges pointwise to $u^{(2)}(x)$. Here, $y_{l,j} \in \mathbb{C}$ is some constant. From the convergence of $u^{(1)}_{\epsilon}$ in $H^{1}(W)$ we obtain that $\int_{W} (n(y)-1) u^{(1)}_{\epsilon}(y)G_{\epsilon}(x,y)dy$ converges $\int_{W} (n(y)-1) u^{(1)}(y)G(x,y)dy$ as $\epsilon \to 0$.
\par
By the argument of (b) in Lemma 6.1 of \cite{Kirsch and Lechleiter2} we have 
\begin{eqnarray}
\lefteqn{\psi_{l,j,\epsilon}(x_1):=\int ^{1/2}_{-1/2}\frac{e^{i\alpha x_1}}{i\epsilon-d_{l,j}\alpha}d\alpha}
\nonumber\\
&=&
-\frac{i}{|d_{l,j}|}\int ^{|d_{l,j}|/(2\epsilon)}_{-|d_{l,j}|/(2\epsilon)}\frac{\mathrm{cos}(t \epsilon x_1/|d_{l,j}|)}{1+t^{2}}dt-2id_{l,j}\int ^{x_1/2}_{0}\frac{t\mathrm{sin}t}{x^2_1\epsilon^2+d_{l,j}^2t^{2}}dt, \ \ \ \ \ \ \ \ \ \ \   \label{2.26} 
\end{eqnarray}
which implies that for all $x_1 \in \mathbb{R}$
\begin{eqnarray}
\lefteqn{\bigl|\psi_{l,j,\epsilon}(x_1)\bigr|\leq C\biggl(\int^{\infty}_{-\infty}\frac{dt}{1+t^2}+\int^{|x_1|/2}_{0}\biggl|\frac{\mathrm{sin}t}{t}\biggr|dt\biggr)}
\nonumber\\
&\leq&
C\biggl(\int^{\infty}_{-\infty}\frac{dt}{1+t^2}dt+\int^{1}_{0}\biggl|\frac{\mathrm{sin}t}{t}\biggr|dt+\int^{|x_1|+1}_{1}\frac{1}{t}dt\biggr)
\nonumber\\
&\leq&
C\bigl(1+\mathrm{log}(|x_1|+1)\bigr), \label{2.27} 
\end{eqnarray}
where above $C$ is independent of $\epsilon>0$. Then, we have that for $y \in W$
\begin{equation}
\bigl|(n(y)-1) u^{(2)}_{\epsilon}(y)G_{\epsilon}(x,y)\bigr| \leq \frac{C\bigl(1+\mathrm{log}(|y_1|+1)\bigr)}{1+|x-y|^{3/2}},\label{2.28}
\end{equation}
where above $C$ is independent of $y$ and $\epsilon$. Then, right hand side of (\ref{2.28}) is an integrable function in $W$ with respect to $y$. Then, by Lebesgue dominated convergence theorem $\int_{W} (n(y)-1) u^{(2)}_{\epsilon}(y)G_{\epsilon}(x,y)dy$ converges to $\int_{W} (n(y)-1) u^{(2)}(y)G(x,y)dy$ as $\epsilon \to 0$. Therefore, (\ref{2.23}) has been shown.
\end{proof}
\section{Uniqueness of $u^{(2)}$}
In Section 3, we will show the uniqueness of $u^{(2)}$ in Theorem 1.2.
\begin{lem}
Let Assumptions 2.1 hold and let $k>0$ be regular in the sense of Definition 2.3. If $u \in H^{1}_{loc}(\mathbb{R}^2_{+})$ such that
\begin{equation}
\Delta u+k^2(1+q)nu=0, \ \mathrm{in} \ \mathbb{R}^2_{+}, \label{3.1}
\end{equation}
\begin{equation}
u=0 \ \mathrm{on} \ \Gamma_0, \label{3.2}
\end{equation}
and $u$ satisfies the radiation condition in the sense of Definition 2.4, then $u^{(2)}=0$ in $\mathbb{R}^2_{+}$.
\end{lem}
\begin{proof}[{\bf Proof of Lemma 3.1}]
By the definition of the radiation condition, $u$ is of the form $u=u^{(1)}+u^{(2)}$ where $u^{(1)} \bigl|_{\mathbb{R} \times (0,R)} \in H^{1}(\mathbb{R} \times (0,R))$ for all $R>0$, and $u^{(2)}\in L^{\infty}(\mathbb{R}^{2}_{+})$ has the  form
\begin{equation}
u^{(2)}(x)=\psi^{+}(x_1)\sum_{j \in J} \sum_{d_{l,j}>0}a_{l,j}\phi_{l,j}(x)+\psi^{-}(x_1)\sum_{j \in J} \sum_{d_{l,j}<0}a_{l,j}\phi_{l,j}(x), \label{3.3}
\end{equation}
where some $a_{l,j} \in \mathbb{C}$, and $\{d_{l,j},\phi_{l,j}: l=1,...,m_j \}$ are normalized eigenvalues and eigenfunctions of the problem (\ref{2.13}). Here, by Remark 2.5 the function $\psi^{+}$ is chosen as a smooth function such that $\psi^{+}(x_1)=1$ for $x_1\geq \eta$ and $\psi^{+}(x_1)=0$ for $x_1\leq -\eta$, and $\psi^{-}:=1-\psi^{+}$ where $\eta>0$ is some positive number.
\vspace{3mm}\\
{\bf Step1} (Green's theorem in $\Omega_N$): We set $\Omega_N:=(-N,N) \times (0, \phi(N))$ where $\psi(N):=N^{s}$. Later we will choose a appropriate $s \in (0,1)$. Let $R>h$ be large and always fixed, and let $N$ be large enough such that $\phi(N)>R$. We denote by $I_{\pm N}^{R}:=\{\pm N \}\times (0,R)$, $I_{\pm N}^{\phi(N)}:=\{\pm N \}\times (R,\phi(N))$, and $\Gamma_{\phi(N), N}:=(-N,N)\times \{\phi(N) \}$. (see the figure below.) We set $I_{\pm N}:=I_{\pm N}^{R} \cup I_{\pm N}^{\phi(N)}$.  
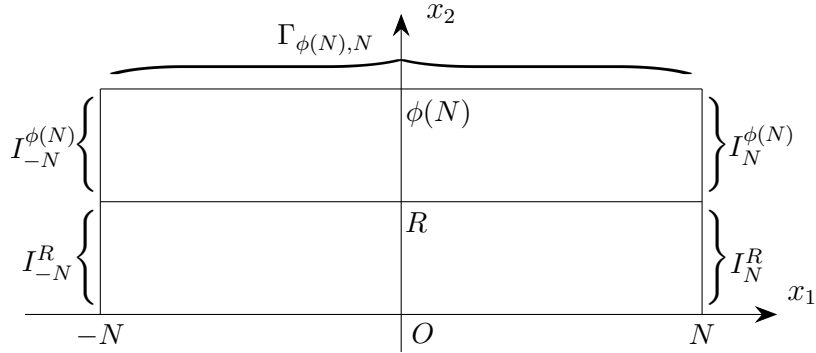
\par \vspace{3mm}
\begin{tikzpicture}
        \path[draw,-{Stealth[length=3mm]}] (-5, 0) -- (5,0) node[above right] {\large $x_1$};
        \path[draw,-{Stealth[length=3mm]}] (0, -0.5) -- (0,4) node[right=2mm] {\large $x_2$} ;
       
        \coordinate (O) at (0,-0.25) node at (O) [right] {$O$};

\draw (4,3) -- (4,0) node [below] {$N$};
\draw (-4,3) -- (-4,0) node [below] {$-N$};
\draw (-4,3) -- (0,3);
\draw (0,3) -- (4,3);
\draw (-4,1.5) -- (0,1.5);
\draw (0,1.5) -- (4,1.5);
\node (A) at (0.5,3) [below] {$\phi(N)$};
\node (B) at (0.2,1.5) [below] {$R$};
\node (C) at (-1,4) [below] {$\Gamma_{\phi(N), N}$};
\node (D) at (0,3) [above] {\scalebox{6.7}[1]{\rotatebox{270}{$\Biggl\{$}}};
\node (E) at (-4.2,0) [above] {{\large$ \Biggl\{$}};
\node (F) at (-4.2,1.5) [above] {{\large$ \Biggl\{$}};
\node (G) at (4.2,0) [above] {{\large$ \Biggr\}$}};
\node (H) at (4.2,1.5) [above] {{\large$ \Biggr\}$}};
\node (I) at (-4.7,0.3) [above] {$I_{-N}^{R}$};
\node (J) at (-4.75,1.8) [above] {$I_{-N}^{\phi(N)}$};
\node (K) at (4.6,0.3) [above] {$I_{N}^{R}$};
\node (L) at (4.8,1.85) [above] {$I_{N}^{\phi(N)}$};
\end{tikzpicture} 
\vspace{5mm}\par
By Green's first theorem in $\Omega_N$ and $u=0$ on $(-N,N)\times \{ 0\}$, we have 
\begin{eqnarray}
\lefteqn{ \int_{\Omega_N}\{-k^2(1+q)n|u|^{2}+|\nabla u|^{2} \}dx=\int_{\Omega_N}\{ \overline{u}\Delta u+|\nabla u|^{2} \}dx}
\nonumber\\
&=&\int_{I_{N}} \overline{u}\frac{\partial u}{\partial x_1}  ds-\int_{I_{-N}} \overline{u}\frac{\partial u}{\partial x_1} ds +\int_{\Gamma_{\phi(N),N}} \overline{u}\frac{\partial u}{\partial x_2} ds
\nonumber\\
&=&\int_{I_{N}} \overline{u^{(2)}}\frac{\partial u^{(2)}}{\partial x_1} ds-\int_{I_{-N}} \overline{u^{(2)}}\frac{\partial u^{(2)}}{\partial x_1} ds\nonumber
\end{eqnarray}
\begin{eqnarray}
&+&\int_{I_{N}} \overline{u^{(1)}}\frac{\partial u^{(1)}}{\partial x_1} ds+\int_{I_{N}} \overline{u^{(1)}}\frac{\partial u^{(2)}}{\partial x_1} ds+\int_{I_{N}} \overline{u^{(2)}}\frac{\partial u^{(1)}}{\partial x_1} ds
\nonumber\\
&-&\int_{I_{-N}} \overline{u^{(1)}}\frac{\partial u^{(1)}}{\partial x_1} ds-\int_{I_{-N}} \overline{u^{(1)}}\frac{\partial u^{(2)}}{\partial x_1} ds-\int_{I_{-N}} \overline{u^{(2)}}\frac{\partial u^{(1)}}{\partial x_1} ds
\nonumber\\
&+&\int_{\Gamma_{\phi(N),N}} \overline{u}\frac{\partial u}{\partial x_2} ds.\label{3.4}
\end{eqnarray}
By the same argument in Theorem 4.6 of \cite{Kirsch and Lechleiter1} and Lemma 6.3 of \cite{Kirsch and Lechleiter2}, we can show that
\begin{eqnarray}
\lefteqn{ \int_{I_{N}} \overline{u^{(2)}}\frac{\partial u^{(2)}}{\partial x_1} ds-\int_{I_{-N}} \overline{u^{(2)}}\frac{\partial u^{(2)}}{\partial x_1} ds }
\nonumber\\
&+&\int_{I_{N}^{R}} \overline{u^{(1)}}\frac{\partial u^{(1)}}{\partial x_1} ds+\int_{I_{N}^{R}} \overline{u^{(1)}}\frac{\partial u^{(2)}}{\partial x_1} ds+\int_{I_{N}^{R}} \overline{u^{(2)}}\frac{\partial u^{(1)}}{\partial x_1} ds
\nonumber\\
&-&\int_{I_{-N}^{R}} \overline{u^{(1)}}\frac{\partial u^{(1)}}{\partial x_1} ds-\int_{I_{-N}^{R}} \overline{u^{(1)}}\frac{\partial u^{(2)}}{\partial x_1} ds-\int_{I_{-N}^{R}} \overline{u^{(2)}}\frac{\partial u^{(1)}}{\partial x_1} ds
\nonumber\\
&=&\frac{1}{2 \pi}\sum_{j \in J} \sum_{d_{l,j},d_{l',j}>0}\overline{a_{l,j}}a_{l',j}\int_{C_{\phi(N)}}\overline{\phi_{l,j}}\frac{\partial \phi_{l',j}}{\partial x_1}dx
\nonumber\\
&-&\frac{1}{2 \pi}\sum_{j \in J} \sum_{d_{l,j},d_{l',j}<0}\overline{a_{l,j}}a_{l',j}\int_{C_{\phi(N)}}\overline{\phi_{l,j}}\frac{\partial \phi_{l',j}}{\partial x_1}dx+o(1),\label{3.5}
\end{eqnarray}
and the first and second term in the right hand side converge as $N \to \infty$ to $\frac{ik}{2\pi}\sum_{j \in J}\sum_{d_{l,j}>0}|a_{l,j}|^{2}d_{l,j}$  and $-\frac{ik}{2\pi}\sum_{j \in J}\sum_{d_{l,j}<0}|a_{l,j}|^{2}d_{l,j}$ respectively. Therefore, taking an imaginary part in (\ref{3.4}) yields that
\begin{eqnarray}
\lefteqn{0=\mathrm{Im}\Biggl[\frac{1}{2\pi} \sum_{j \in J} \sum_{d_{l,j},d_{l',j}>0}\overline{a_{l,j}}a_{l',j}\int_{C_{\phi(N)}}\overline{\phi_{l,j}}\frac{\partial \phi_{l',j}}{\partial x_1}dx \Biggr]}
\nonumber\\
&-&\mathrm{Im}\Biggl[\frac{1}{2\pi} \sum_{j \in J}\sum_{d_{l,j},d_{l',j}<0}\overline{a_{l,j}}a_{l',j}\int_{C_{\phi(N)}}\overline{\phi_{l,j}}\frac{\partial \phi_{l',j}}{\partial x_1}dx\Biggr]
\nonumber\\
&+&\mathrm{Im}\int_{I_{N}^{\phi(N)}} \overline{u^{(1)}}\frac{\partial u^{(1)}}{\partial x_1} ds+\mathrm{Im}\int_{I_{N}^{\phi(N)}} \overline{u^{(1)}}\frac{\partial u^{(2)}}{\partial x_1} ds+\mathrm{Im}\int_{I_{N}^{\phi(N)}} \overline{u^{(2)}}\frac{\partial u^{(1)}}{\partial x_1} ds
\nonumber\\
&-&\mathrm{Im}\int_{I_{-N}^{\phi(N)}} \overline{u^{(1)}}\frac{\partial u^{(1)}}{\partial x_1} ds-\mathrm{Im}\int_{I_{-N}^{\phi(N)}} \overline{u^{(1)}}\frac{\partial u^{(2)}}{\partial x_1} ds-\mathrm{Im}\int_{I_{-N}^{\phi(N)}} \overline{u^{(2)}}\frac{\partial u^{(1)}}{\partial x_1} ds
\nonumber\\
&+&\mathrm{Im}\int_{\Gamma_{\phi(N),N}} \overline{u}\frac{\partial u}{\partial x_2}ds+o(1).\label{3.6}
\end{eqnarray}
We set 
\begin{equation}
J_{\pm}(N):=\pm \mathrm{Im}\int_{I_{\pm N}^{\phi(N)}} \overline{u^{(1)}}\frac{\partial u^{(1)}}{\partial x_1} ds\pm \mathrm{Im}\int_{I_{\pm N}^{\phi(N)}} \overline{u^{(1)}}\frac{\partial u^{(2)}}{\partial x_1} ds\pm \mathrm{Im}\int_{I_{\pm N}^{\phi(N)}} \overline{u^{(2)}}\frac{\partial u^{(1)}}{\partial x_1} ds,\label{3.7}
\end{equation}
and we will show that $\mathrm{limsup_{N\to \infty}}J_{\pm}(N)\geq0$.
\vspace{2mm}\\
{\bf Step2 ($\mathrm{limsup_{N\to \infty}}J_{\pm}(N)\geq0$):}
By Cauchy Schwarz inequality we have 
\begin{eqnarray}
\lefteqn{|J_{+}(N)|\leq \biggl(\int^{\phi(N)}_{R} |u^{(1)}(N,x_2)|^2dx_2\biggr)^{1/2}\biggl(\int^{\phi(N)}_{R} \biggl|\frac{\partial u^{(1)}}{\partial x_1} (N,x_2)\biggr|^2dx_2\biggr)^{1/2}}
\nonumber\\
&+&\biggl(\int^{\phi(N)}_{R} |u^{(1)}(N,x_2)|^2dx_2\biggr)^{1/2}\biggl(\int^{\phi(N)}_{R} \biggl|\frac{\partial u^{(2)}}{\partial x_1} (N,x_2)\biggr|^2dx_2\biggr)^{1/2}
\nonumber\\
&+&\biggl(\int^{\phi(N)}_{R} |u^{(2)}(N,x_2)|^2dx_2\biggr)^{1/2}\biggl(\int^{\phi(N)}_{R} \biggl|\frac{\partial u^{(1)}}{\partial x_1} (N,x_2)\biggr|^2dx_2\biggr)^{1/2}
\nonumber\\
&\leq&\biggl(\int^{\phi(N)}_{R} |u^{(1)}(N,x_2)|^2dx_2\biggr)^{1/2}\biggl(\int^{\phi(N)}_{R} \biggl|\frac{\partial u^{(1)}}{\partial x_1} (N,x_2)\biggr|^2dx_2\biggr)^{1/2}
\nonumber\\
&+&C(\phi(N)-R)^{1/2}\biggl(\int^{\phi(N)}_{R} |u^{(1)}(N,x_2)|^2dx_2\biggr)^{1/2}
\nonumber\\
&+&C(\phi(N)-R)^{1/2}\biggl(\int^{\phi(N)}_{R} \biggl|\frac{\partial u^{(1)}}{\partial x_1} (N,x_2)\biggr|^2dx_2\biggr)^{1/2}.\label{3.8}
\end{eqnarray}
In order to estimate $u^{(1)}$, we will show the following lemma.
\begin{lem}
$u^{(1)}$ has an integral representation of the form
\begin{eqnarray}
u^{(1)}(x)&=&\int_{y_2>0}\sigma(y)G(x,y)dy+k^{2}\int_{W} \bigl(n(y)(1+q(y))-1\bigr) u^{(1)}(y)G(x,y)dy
\nonumber\\
&+&k^2\int_{Q} n(y)q(y)u^{(2)}(y)G(x,y)dy, \ \ x_2>0,\label{3.9} 
\end{eqnarray}
where $\sigma:=\Delta u^{(2)}+k^2nu^{(2)}$.
\end{lem}
\begin{proof}[Proof of Lemma 3.2]
First, we will consider an integral representation of $u^{(2)}$. Let $N>0$ be large enough. By Green's representation theorem in $(-N,N)\times (0,N^{1/4})$, we have
\begin{eqnarray}
u^{(2)}(x)&=&\int_{(-N,N) \times \{N^{1/4} \}} \bigl[u^{(2)}(y)\frac{\partial G}{\partial y_2}(x,y)-G(x,y)\frac{\partial u^{(2)}}{\partial y_2}(y)\bigr]ds(y)
\nonumber
\end{eqnarray}
\begin{eqnarray}
&+&\biggl(\int_{\{N \} \times (0, N^{1/4})}-\int_{\{-N \} \times (0, N^{1/4})}\biggr) \bigl[u^{(2)}(y)\frac{\partial G}{\partial y_1}(x,y)-G(x,y)\frac{\partial u^{(2)}}{\partial y_1}(y)\bigr]ds(y)
\nonumber\\
&-&\int_{(-N,N) \times (0,N^{1/4})}\bigl[\sigma(y)+k^2(1-n(y))u^{(2)}(y)\bigr]G(x,y)dy.\label{3.10}
\end{eqnarray}
By Lemma 3.1 of \cite{Chandler and Christopher}, the Dirichlet Green's function $G(x,y)$ is of the estimation
\begin{equation}
|G(x,y)|, \ |\nabla_y G(x,y)| \leq C \frac{x_2 y_2}{1+|x-y|^{3/2}}, \ |x-y|>1. \label{3.11} 
\end{equation}
By Lemma 2.2 we have that $|u^{(2)}(x)|, \ \bigl|\frac{\partial u^{(2)}(x)}{\partial x_2}\bigr| \leq ce^{-\delta |x_2|}$ for all $x \in \mathbb{R}^{2}_{+}$, and some $c,\delta >0$. Then, we obtain 
\begin{eqnarray}
\lefteqn{\Biggl|\int_{(-N,N) \times \{N^{1/4} \}} \bigl[u^{(2)}(y)\frac{\partial G}{\partial y_2}(x,y)-G(x,y)\frac{\partial u^{(2)}}{\partial y_2}(y)\bigr]ds(y)\Biggr|}
\nonumber\\
&\leq&C \int_{-N}^{N} \frac{x_2 e^{-\delta N^{1/4}}}{|N^{1/4}-x_2|^{3/2}} dy_2 \leq C\frac{x_2 N e^{-\delta N^{1/4}}}{|N^{1/4}-x_2|^{3/2}}.\label{3.12} 
\end{eqnarray}
Furthermore,
\begin{eqnarray}
\lefteqn{\Biggl|\int_{\{\pm N \} \times (0, N^{1/4}) } \bigl[u^{(2)}(y)\frac{\partial G}{\partial y_1}(x,y)-G(x,y)\frac{\partial u^{(2)}}{\partial y_1}(y)\bigr]ds(y)\Biggr|}
\nonumber\\
&\leq&C \int_{0}^{N^{1/4}} \frac{x_2 y_2}{|\pm N-x_1|^{3/2}} dy_2 \leq C\frac{x_2 N^{1/2}}{|\pm N-x_1|^{3/2}}.\label{3.13} 
\end{eqnarray}
Therefore, as $N \to \infty$ in (\ref{3.10}) we get
\begin{eqnarray}
u^{(2)}(x)=-\int_{y_2>0}\sigma(y)G(x,y)dy+k^2\int_{W}(n(y)-1)u^{(2)}(y)G(x,y)dy. \ \ \ \label{3.14} 
\end{eqnarray}
By Lemma 2.7, we have (substitute $-k^2qnu$ for $f$ in (\ref{2.18}))
\begin{equation}
u(x)=k^2\int_{W}\bigl(n(y)-1\bigr) u(y)G(x,y)dy+k^2\int_{Q}q(y)n(y)u(y)G(x,y)dy. \label{3.15}
\end{equation}
Combining (\ref{3.14}) with (\ref{3.15}) we have
\begin{eqnarray}
u^{(1)}(x)&=&-u^{(2)}(x)+k^2\int_{W}\bigl(n(y)-1\bigr) u(y)G(x,y)dy+k^2\int_{Q}q(y)n(y)u(y)G(x,y)dy
\nonumber
\end{eqnarray}
\begin{eqnarray}
&=&\int_{y_2>0}\sigma(y) G(x,y)dy-k^2\int_{W}(n(y)-1)u^{(2)}(y)G(x,y)dy
\nonumber\\
&+&k^2\int_{W}\bigl(n(y)-1\bigr) u(y)G(x,y)dy+k^2\int_{Q}q(y)n(y)u(y)G(x,y)dy
\nonumber\\
&=&\int_{\mathbb{R}_{+}^{2}}\sigma(y) G(x,y)dy+k^{2}\int_{W} \bigl( n(y)(1+q(y))-1\bigr) u^{(1)}(y)G(x,y)dy
\nonumber\\
&+&k^2\int_{Q}n(y)q(y)u^{(2)}(y)G(x,y)dy.\label{3.16} 
\end{eqnarray}
Therefore, Lemma 3.2 has been shown.
\end{proof}
\vspace{0.5cm}
We set $u^{\pm}(x):=\sum_{j \in J} \sum_{d_{l,j}\lessgtr 0}a_{l,j}\phi_{l,j}(x)$. Then, by simple calculation we can show
\begin{equation}
\sigma(y)=\frac{d^{2} \psi^{+}(y_1)}{d y^{2}_1}u^{+}(y)+2\frac{d \psi^{+}(y_1)}{d y_1}\frac{\partial u^{+}(y)}{\partial y_1}+\frac{d^{2} \psi^{-}(y_1)}{d y^{2}_1}u^{-}(y)+2\frac{d \psi^{-}(y_1)}{d y_1}\frac{\partial u^{-}(y)}{\partial y_1},\label{3.17} 
\end{equation}
which implies that $\mathrm{supp}\sigma \subset (-\eta, \eta)\times (0,\infty)$. By Lemma 3.2 we have for $R<x_2<\phi(N)$
\begin{eqnarray}
\lefteqn{|u^{(1)}(N,x_2)|, \ \biggl|\frac{\partial u^{(1)}}{\partial x_1} (N,x_2)\biggr|\leq C\int_{(-\eta, \eta)\times (0,\infty)} |\sigma(y)|\frac{\phi(N) y_2}{|N-\eta|^{3/2}}dy}
\nonumber\\
&+&
C\int_{W} |u^{(1)}(y)|\frac{\phi(N) h}{(1+|N-y_1|)^{3/2}}dy
+C\int_{Q} \frac{\phi(N)|u^{(2)}(y)|}{|N-y_1|^{3/2}}dy
\nonumber\\
&\leq&C\frac{\phi(N)}{N^{3/2}}+C\phi(N)\int_{W} \frac{|u^{(1)}(y)|}{(1+|N-y_1|)^{3/2}}dy.\label{3.18} 
\end{eqnarray}
We have to estimate the second term in right hand side. The following lemma was shown in Lemma 4.12 of \cite{Chandler}.
\begin{lem}
Assume that $\varphi \in L^{2}_{loc}(\mathbb{R})$ such that 
\begin{equation}
\mathrm{sup}_{A>0}\Bigl\{ (1+A^{2})^{-\epsilon}\int^{A}_{-A}|\varphi(t)|^{2}dt\Bigr\}< \infty, \label{3.19} 
\end{equation}
for some $\epsilon>0$. Then, for every $\alpha \in [0,\frac{1}{2}-\epsilon)$ there exists a constant $C>0$ and a sequence $\{A_m \}_{m\in \mathbb{N}}$ such that $A_m \to \infty$ as $m \to \infty$ and 
\begin{equation}
\int_{K_{A_m}}|\varphi(t)|^{2}dt \leq C A^{-\alpha}_m, \ m \in\mathbb{N},\label{3.20} 
\end{equation}
where $K_A:=K^{+}_{A}\cup K^{-}_{A}$, $K^{+}_{A}:=(-A^{+},A^{+})\setminus(-A,A)$, $K^{-}_{A}:=(-A,A)\setminus(-A^{-},A^{-})$, and $A^{\pm}:=A\pm A^{1/2}$ for $A \in [1, \infty)$.
\end{lem}
Applying Lemma 3.3 to $\varphi=\bigl(\int^{h}_{0} \bigl|u^{(1)} (\cdot,y_2)\bigr|^2dy_2\bigr)^{1/2} \in L^{2}(\mathbb{R})$, there exists a sequence $\{N_m \}_{m\in \mathbb{N}}$ such that $N_m \to \infty$ as $m \to \infty$ and 
\begin{equation}
\int_{K_{N_m}}\int_{0}^{h}|u^{(1)} (y_1,y_2)|^{2}dy_1dy_2 \leq C N^{-1/4}_m, \ m \in\mathbb{N}.\label{3.21} 
\end{equation}
Then, by Cauchy Schwarz inequality we have 
\begin{eqnarray}
\lefteqn{\int_{W} \frac{|u^{(1)}(y)|}{(1+|N-y_1|)^{3/2}}dy=\biggl(\int_{-N_{m}^{-}}^{N_{m}^{-}}+\int_{K_{N_{m}}}+\int_{\mathbb{R}\setminus [-N_m^+,N_m^+]}\biggr)\int_{0}^{h} \frac{|u^{(1)}(y)|}{(1+|N_m-y_1|)^{3/2}}dy}
\nonumber\\
&\leq&C\biggl(\int_{-N_{m}^{-}}^{N_{m}^{-}}\frac{dy_1}{(1+N_m-|y_1|)^{3}}\biggr)^{1/2}+C\biggl(\int_{K_{N_m}}\int_{0}^{h}|u^{(1)} (y_1,y_2)|^{2}dy_1dy_2 \biggr)^{1/2}
\nonumber\\
&+&C\biggl(\int_{\mathbb{R}\setminus [-N_m^+,N_m^+]}\frac{dy_1}{(1+|y_1|-N_m)^{3}}\biggr)^{1/2}
\nonumber\\
&\leq&C\biggl(\int_{0}^{N_{m}^{-}}\frac{dy_1}{(1+N_m-y_1)^{3}}\biggr)^{1/2}+CN^{-1/8}_m+C\biggl(\int_{N_{m}^+}^{\infty}\frac{dy_1}{(1+y_1-N_m)^{3}}\biggr)^{1/2}
\nonumber\\
&\leq&CN^{-1/8}_m.\label{3.22} 
\end{eqnarray}
With (\ref{3.18}) we have for $m \in \mathbb{N}$,
\begin{eqnarray}
|u^{(1)}(N_m,x_2)|, \ \biggl|\frac{\partial u^{(1)}}{\partial x_1} (N_m,x_2)\biggr|
\leq C\frac{\phi(N_m)}{N_m^{1/8}}.\label{3.23} 
\end{eqnarray}
Therefore, by (\ref{3.8}) we have 
\begin{eqnarray}
|J_{+}(N_m)|
&\leq&C(\phi(N_m)-R)\frac{\phi(N_m)^{2}}{N_m^{1/4}}
+C(\phi(N_m)-R)\frac{\phi(N_m)}{N_m^{1/8}}
\nonumber\\
&\leq&C(\phi(N_m)-R)\frac{\phi(N_m)^{2}}{N_m^{1/8}}\leq C\frac{\phi(N_m)^{3}}{N_m^{1/8}}.\label{3.24}
\end{eqnarray}
Since $\phi(N)=N^s$, if we choose $s \in (0,1)$ such that $3s<\frac{1}{8}$, that is, $0<s<\frac{1}{24}$ the right hand side in (\ref{3.24}) converges to zero as $m \to \infty$. Therefore, $\mathrm{limsup_{N\to \infty}}J_{+}(N)\geq0$. By the same argument of $J_{+}$, we can show that $\mathrm{limsup_{N\to \infty}}J_{-}(N)\geq0$, which yields Step 2.
\vspace{5mm}\\
Next, we discuss the last term in (\ref{3.6}). By the same argument in Lemma 3.2 that we apply Green's representation theorem in $x_2>h$ and use the Dirichlet Green's function $G_h$ of $\mathbb{R}^{2}_{x_2>h}(:=\mathbb{R}\times (h, \infty))$ insted of $G$, $u^{(1)}$ can also be of another integral representation for $x_2>h$
\begin{eqnarray}
u^{(1)}(x)&=&\int_{y_2>h}\sigma(y)G_h(x,y)dy+2\int_{\Gamma_h} u^{(1)}(y)\frac{\partial \Phi_k(x,y)}{\partial y_2}ds(y)
\nonumber\\
&=:&v^{1}(x)+v^{2}(x),\label{3.25}
\end{eqnarray}
where $G_h$ is defined by $G_h(x,y):=\Phi_k(x,y)-\Phi_k(x,y^*_h)$ where $y^*_h=(y_1, 2h-y_2)$. We define approximation $u^{(1)}_N$ of $u^{(1)}$ by 
\begin{eqnarray}
u^{(1)}_N(x)&:=&\int_{y_2>0}\chi_{\phi(N)-1}(y_2)\sigma(y)G(x,y)dy+2\int_{\Gamma_h} \chi_{N}(y_1)u^{(1)}(y)\frac{\partial \Phi_k(x,y)}{\partial y_2}ds(y)
\nonumber\\
&=:&v^{1}_{N}(x)+v^{2}_{N}(x), \ \ \  x_2>h,\label{3.26}
\end{eqnarray}
where $\chi_a$ is defined by for $a>0$,
\begin{equation}
\chi_{a}(t):=\left\{ 
       \begin{array}{rl}
       1 & \quad \mbox{for $|t|\leq a$} \\
       0 & \quad \mbox{for $|t|> a$}.
       \end{array}\right.\label{3.27}
\end{equation}
By Lemma 3.4 of \cite{Chandler and Zhang2} and Lemma 2.1 of \cite{Chandler and Zhang1} we can show that $v^{1}_{N}$ and $v^{2}_{N}$ satisfy the upward propagating radiation condition, which implies that so does $u^{(1)}_N$. Furthermore, by the definition of $u^{(1)}_N$ we can show that $u^{(1)}_N(\cdot, \phi(N)-1) \in L^{2}(\mathbb{R})\cap L^{\infty}(\mathbb{R})$. Then, by Lemma 6.1 of \cite{Chandler and Zhang2} we have that
\begin{equation}
\mathrm{Im}\int_{\Gamma_{\phi(N)}}\overline{u^{(1)}_N}\frac{\partial u^{(1)}_N}{\partial x_2}ds \geq0.\label{3.28}
\end{equation}
Combining (\ref{3.6}) with (\ref{3.28}) we have
\begin{eqnarray}
\lefteqn{0\geq-\mathrm{Im}\int_{\Gamma_{\phi(N)}}\overline{u^{(1)}_N}\frac{\partial u^{(1)}_N}{\partial x_2}ds}
\nonumber\\
&=&\mathrm{Im}\Biggl[ \frac{1}{2\pi}\sum_{j \in J} \sum_{d_{l,j},d_{l',j}>0}\overline{a_{l,j}}a_{l',j}\int_{C_{\phi(N)}}\overline{\phi_{l,j}}\frac{\partial \phi_{l,j}}{\partial x_1}dx \Biggr]
\nonumber\\
&-&
\mathrm{Im}\Biggl[\frac{1}{2\pi}\sum_{j \in J} \sum_{d_{l,j},d_{l',j}<0}\overline{a_{l,j}}a_{l',j}\int_{C_{\phi(N)}}\overline{\phi_{l,j}}\frac{\partial \phi_{l,j}}{\partial x_1}dx\Biggr]+J_{+}(N)+J_{-}(N)
\nonumber\\
&+&
\mathrm{Im}\int_{\Gamma_{\phi(N),N}} \overline{u}\frac{\partial u}{\partial x_2}-\mathrm{Im}\int_{\Gamma_{\phi(N)}}\overline{u^{(1)}_N}\frac{\partial u^{(1)}_N}{\partial x_2}ds+o(1).\label{3.29}
\end{eqnarray}
We observe the last term
\begin{equation}
\mathrm{Im}\int_{\Gamma_{\phi(N),N}} \overline{u}\frac{\partial u}{\partial x_2}-\mathrm{Im}\int_{\Gamma_{\phi(N)}}\overline{u^{(1)}_N}\frac{\partial u^{(1)}_N}{\partial x_2}ds=:L(N)+M(N),\label{3.30}
\end{equation}
where
\begin{equation}
L(N):=\mathrm{Im}\int_{\Gamma_{\phi(N),N}}\overline{u^{(1)}}\frac{\partial u^{(1)}}{\partial x_2}ds-\mathrm{Im}\int_{\Gamma_{\phi(N)}}\overline{u^{(1)}_N}\frac{\partial u^{(1)}_N}{\partial x_2}ds,\label{3.31}
\end{equation}
\begin{equation}
M(N):=\mathrm{Im}\int_{\Gamma_{\phi(N),N}}\overline{u^{(1)}}\frac{\partial u^{(2)}}{\partial x_2}ds+\mathrm{Im}\int_{\Gamma_{\phi(N),N}}\overline{u^{(2)}}\frac{\partial u^{(1)}}{\partial x_2}ds+\mathrm{Im}\int_{\Gamma_{\phi(N),N}}\overline{u^{(2)}}\frac{\partial u^{(2)}}{\partial x_2}ds.\label{3.32}
\end{equation}
By Lemma 3.2 we can show $|u^{(1)}(x_1, \phi(N))|$, $|\frac{\partial u^{(1)}}{\partial x_2}(x_1,\phi(N))|\leq C\phi(N)$ for $x_1\in \mathbb{R}$, and by Lemma 2.2 we have $|u^{(2)}(x_1, \phi(N))|$, $|\frac{\partial u^{(2)}}{\partial x_2}(x_1, \phi(N))|\leq Ce^{-\delta \phi(N)}$ for $x_1\in \mathbb{R}$. Then, we have
\begin{eqnarray}
|M(N)|
&\leq&
\int^{N}_{-N}|u^{(1)}(x_1, \phi(N))|\Bigl|\frac{\partial u^{(2)}}{\partial x_2}(x_1,\phi(N))\Bigr|dx_1
\nonumber\\
&+&
\int^{N}_{-N}|u^{(2)}(x_1, \phi(N))|\Bigl|\frac{\partial u^{(1)}}{\partial x_2}(x_1,\phi(N))\Bigr|dx_1
\nonumber\\
&+&\int^{N}_{-N}|u^{(2)}(x_1, \phi(N))|\Bigl|\frac{\partial u^{(2)}}{\partial x_2}(x_1,\phi(N))\Bigr|dx_1
\nonumber\\
&\leq&C(N\phi(N)e^{-\delta \phi(N)}+Ne^{-2\delta \phi(N)})
\nonumber\\
&\leq&
CN\phi(N)e^{-\delta \phi(N)},
\label{3.33}
\end{eqnarray}
which implies that $M(N)=o(1)$ as $N \to \infty$. Hence, we will show that $\mathrm{limsup_{N\to \infty}}L(N)\geq0$.\vspace{1mm}\\
{\bf Step3 ($\mathrm{limsup_{N\to \infty}}L(N)\geq0$):}
First, we observe that
\begin{eqnarray}
|L(N)|&\leq&
\biggl|\mathrm{Im}\int_{\Gamma_{\phi(N),N}}\overline{u^{(1)}}\frac{\partial u^{(1)}}{\partial x_2}ds-\mathrm{Im}\int_{\Gamma_{\phi(N),N}}\overline{u^{(1)}}\frac{\partial u^{(1)}_N}{\partial x_2}ds\biggr|
\nonumber\\
&+&\biggl| \mathrm{Im}\int_{\Gamma_{\phi(N)},N}\overline{u^{(1)}}\frac{\partial u^{(1)}_N}{\partial x_2}ds-\mathrm{Im}\int_{\Gamma_{\phi(N)},N}\overline{u^{(1)}_N}\frac{\partial u^{(1)}_N}{\partial x_2}ds\biggr|
\nonumber\\
&+&\biggl| \mathrm{Im}\int_{\Gamma_{\phi(N)}\setminus\Gamma_{\phi(N),N} }\overline{u^{(1)}_N}\frac{\partial u^{(1)}_N}{\partial x_2}ds\biggr|
\nonumber
\end{eqnarray}
\begin{eqnarray}
&\leq&
\int_{-N}^{N}|u^{(1)}(x_1,\phi(N))|\Bigl|\frac{\partial u^{(1)}}{\partial x_2}(x_1,\phi(N))-\frac{\partial u^{(1)}_N}{\partial x_2}(x_1,\phi(N))\Bigr|ds\nonumber\\
&+&\int_{-N}^{N}|u^{(1)}(x_1,\phi(N))-u^{(1)}_N(x_1,\phi(N))|\Bigl|\frac{\partial u^{(1)}_N}{\partial x_2}(x_1,\phi(N))\Bigr|ds
\nonumber\\
&+&\int_{\mathbb{R}\setminus (-N,N)}|u^{(1)}_N(x_1,\phi(N))|\Bigl|\frac{\partial u^{(1)}_N}{\partial x_2}(x_1,\phi(N))\Bigr|ds\label{3.34}.
\end{eqnarray}
By Lemma 2.2 $\sigma$ has a exponential decay in $y_2$. Then, we have for $x_1\in \mathbb{R}$,
\begin{eqnarray}
\lefteqn{|v^{1}(x_1,\phi(N))|, \ \biggl|\frac{\partial v^{1}}{\partial x_2} (x_1,\phi(N))\biggr|, \ |v^{1}_N(x_1,\phi(N))|, \ \biggl|\frac{\partial v^{1}_N}{\partial x_2} (x_1,\phi(N))\biggr|}
\nonumber\\
&\leq&
C\int_{(-\eta, \eta)\times(0,\infty)} \frac{e^{-\delta y_2}\phi(N) y_2}{(1+|x_1-y_1|)^{3/2}}dy\leq C\frac{\phi(N)}{(1+|x_1|)^{3/2}},\label{3.35}
\end{eqnarray}
and
\begin{eqnarray}
\lefteqn{|v^{1}(x_1,\phi(N))-v^{1}_N(x_1,\phi(N))|, \ \biggl|\frac{\partial v^{1}}{\partial x_2} (x_1,\phi(N))-\frac{\partial v^{1}_N}{\partial x_2} (x_1,\phi(N))\biggr|}
\nonumber\\
&\leq&
C\int_{(-\eta, \eta)\times(\phi(N)-1,\infty)} \frac{e^{-\delta y_2}\phi(N) y_2}{(1+|x_1-y_1|)^{3/2}}dy
\nonumber\\
&\leq&
C\biggl(\int_{\phi(N)}^{\infty} e^{-\delta y_2}y_2dy_2\biggr)\frac{\phi(N)}{(1+|x_1|)^{3/2}}dy\leq \frac{e^{-\delta \phi(N)}\phi(N)}{(1+|x_1|)^{3/2}}.\label{3.36}
\end{eqnarray}
Since the fundamental solution to Helmholtz equation $\Phi(x,y)$ is of the following estimation (see e.g., \cite{Chandler and Christopher}) for $|x-y|\geq 1$
\begin{equation}
\biggl|\frac{\partial \Phi}{\partial y_2}(x,y)\biggr|\leq C\frac{|x_2-y_2|}{1+|x-y|^{3/2}}, \ \ \ \biggl|\frac{\partial^{2} \Phi}{\partial x_2 \partial y_2}(x,y)\biggr|\leq C\frac{|x_2-y_2|^2}{1+|x-y|^{3/2}}, \label{3.37}
\end{equation}
we can show that for $x_1 \in \mathbb{R}$
\begin{equation}
|v^{2}(x_1,\phi(N))|\leq C\phi(N)W_{\infty}(x_1), \ \ \ |v^{2}_N(x_1,\phi(N))|\leq C\phi(N)W_{N}(x_1),\label{3.38}
\end{equation}
and
\begin{equation}
\biggl|\frac{\partial v^{2}}{\partial x_2}(x_1,\phi(N))\biggr|\leq C\phi(N)^2 W_{\infty}(x_1), \ \ \ \biggl|\frac{\partial v^{2}_N}{\partial x_2}(x_1,\phi(N))\biggr| \leq C\phi(N)^2 W_{N}(x_1),\label{3.39}
\end{equation}
and
\begin{equation}
|v^{2}(x_1,\phi(N))-v^{2}_N(x_1,\phi(N))|\leq C\phi(N)\bigl(W_{\infty}(x_1)- W_{N}(x_1)\bigr),\label{3.40}
\end{equation}
and
\begin{equation}
\biggl|\frac{\partial v^{2}}{\partial x_2}(x_1,\phi(N))-\frac{\partial v^{2}_N}{\partial x_2}(x_1,\phi(N))\biggr|\leq C\phi(N)^2 (W_{\infty}(x_1)- W_{N}(x_1)\bigr),\label{3.41}
\end{equation}
where $W_N$ is defined by for $N \in (0, \infty]$
\begin{equation}
W_N(x_1):=\int_{-N}^{N}\frac{|u^{(1)}(y_1,h)|}{(1+|x_1-y_1|)^{3/2}}dy_1, \ \ \ x_1 \in \mathbb{R}.\label{3.42}
\end{equation}
Using (\ref{3.35})--(\ref{3.41}), we continue to estimate (\ref{3.34}). By Cauchy Schwarz inequality we have 
\begin{eqnarray}
\lefteqn{|L(N)|\leq C\int_{-N}^{N}\Bigl\{\frac{\phi(N)}{(1+|x_1|)^{3/2}}+\phi(N)W_{\infty}(x_1)\Bigr\}}
\nonumber\\
&&\ \ \ \ \ \ \ \ \ \ \ \ \ \ \ \ \ \ \ \ \ \ \ \ \ \ \ \ \ \ \ \ \ \ \times\Bigl\{\frac{\phi(N)e^{-\sigma \phi(N)}}{(1+|x_1|)^{3/2}}+\phi(N)^2\bigl(W_{\infty}(x_1)-W_{N}(x_1)\bigr)\Bigr\}dx_1
\nonumber\\
&+&\int_{-N}^{N}\Bigl\{\frac{\phi(N)e^{-\sigma \phi(N)}}{(1+|x_1|)^{3/2}}+\phi(N)\bigl(W_{\infty}(x_1)-W_{N}(x_1)\bigr)\Bigr\}
\nonumber\\
&&\ \ \ \ \ \ \ \ \ \ \ \ \ \ \ \ \ \ \ \ \ \ \ \ \ \ \ \ \ \ \ \ \ \times \Bigl\{\frac{\phi(N)}{(1+|x_1|)^{3/2}}+\phi(N)^2W_{N}(x_1)\Bigr\}dx_1
\nonumber\\
&+&\int_{\mathbb{R}\setminus (-N,N)}\Bigl\{\frac{\phi(N)}{(1+|x_1|)^{3/2}}+\phi(N)W_{N}(x_1)\Bigr\}\Bigl\{\frac{\phi(N)}{(1+|x_1|)^{3/2}}+\phi(N)^2W_{N}(x_1)\Bigr\}dx_1
\nonumber
\end{eqnarray}
\begin{eqnarray}
&\leq&
C\phi(N)^{3}\int_{-N}^{N}W_{\infty}(x_1)\bigl(W_{\infty}(x_1)-W_{N}(x_1)\bigr)dx_1
\nonumber\\
&+&C\phi(N)^{3}\int_{-N}^{N}\frac{1}{(1+|x_1|)^{3/2}}\bigl(W_{\infty}(x_1)-W_{N}(x_1)\bigr)dx_1
\nonumber\\
&+&C\phi(N)^{2}\int_{\mathbb{R}\setminus (-N,N)}\frac{1}{(1+|x_1|)^{3}}dx_1+C\phi(N)^{2}\int_{\mathbb{R}\setminus (-N,N)}\frac{1}{(1+|x_1|)^{3/2}}W_{N}(x_1)dx_1
\nonumber\\
&+&C\phi(N)^{3}\int_{\mathbb{R}\setminus (-N,N)}|W_{N}(x_1)|^2dx_1+o(1)
\nonumber\\
&\leq&
C\phi(N)^{3}\biggl\{\Bigl(\int_{-N}^{N}\bigl(W_{\infty}(x_1)-W_{N}(x_1)\bigr)^2dx_1\Bigr)^{1/2}+\Bigl(\int_{\mathbb{R}\setminus (-N,N)}W_{N}(x_1)^2dx_1\Bigr)^{1/2}\biggr\}
\nonumber\\
&& \ \ \ \ \ \ \ \ \ \ \ \ \ \ \ \ \ \ \ \ \ \ \ \ \ \ \ \ \ \ \ \ \ \ \ \ \ \ \ \ \ \ \ \ \ \ \ \ \ \ \ \ \ \ \ \ \ \ \ \ \ \ \ \ \ \ \ \ \ \ \ \ \ \ \ \ \ \ \ +o(1).\label{3.43}
\end{eqnarray}
Finally, we will estimate $\bigl(W_{\infty}(x_1)-W_{N}(x_1)\bigr)$ and $W_{N}(x_1)$. Since $u^{(1)}(\cdot,h) \in L^{2}(\mathbb{R})$, by Lemma 3.3 there exists a sequence $\{N_m \}_{m\in \mathbb{N}}$ such that $N_m \to \infty$ as $m \to \infty$ and 
\begin{equation}
\int_{K_{N_m}}|u^{(1)}(y_1,h)|^{2}dy_1 \leq C N^{-\frac{1}{4}}_m, \ m \in\mathbb{N},\label{3.44}
\end{equation}
where $K_A:=K^{+}_{A}\cup K^{-}_{A}$, $K^{+}_{A}:=(-A^{+},A^{+})\setminus(-A,A)$, $K^{-}_{A}:=(-A,A)\setminus(-A^{-},A^{-})$, and $A^{\pm}:=A\pm A^{1/2}$ for $A \in [1, \infty)$.
\par
By Cauchy Schwarz inequality we have for $|x_1|>N_m$, 
\begin{eqnarray}
\int_{-N_{m}^{-}}^{N_{m}^{-}}\frac{|u^{(1)}(y_1,h)|}{(1+|x_1-y_1|)^{3/2}}dy_1
&\leq&\biggl(\int_{-N_{m}^{-}}^{N_{m}^{-}}|u^{(1)}(y_1,h)|^{2}dy_1\biggr)^{1/2}\biggl(\int_{-N_{m}^{-}}^{N_{m}^{-}}\frac{dy_1}{(1+|x_1|-y_1)^{3}}\biggr)^{1/2}
\nonumber\\
&\leq& \frac{C}{1-|x_1|-N^{-}_m},\label{3.45}
\end{eqnarray}
and
\begin{eqnarray}
\int_{K_{N^{-}_m}}\frac{|u^{(1)}(y_1,h)|}{(1+|x_1-y_1|)^{3/2}}dy_1
&\leq&\biggl(\int_{K_{N_m}}|u^{(1)}(y_1,h)|^{2}dy_1\biggr)^{1/2}\biggl(\int_{K^{-}_{N_m}}\frac{dy_1}{(1+|x_1|-y_1)^{3}}\biggr)^{1/2}
\nonumber\\
&\leq& \frac{C}{N^{1/8}_{m}(1+|x_1|-N_m)}.\label{3.46}
\end{eqnarray}
Therefore, we obtain 
\begin{eqnarray}
\lefteqn{\int_{\mathbb{R}\setminus (-N_m,N_m)}W_{N}(x_1)^2dx_1}
\nonumber\\
&\leq&C\int_{N_m}^{\infty}\frac{dx_1}{(1-|x_1|-N^{-}_m)^2}+\frac{C}{N_m^{1/4}}\int_{N_m}^{\infty}\frac{dx_1}{(1-|x_1|-N_m)^2}
\nonumber\\
&\leq&\frac{C}{1+N^{1/2}_m}+\frac{C}{N^{1/4}_{m}}\ \leq \ \frac{C}{N^{1/4}_{m}}.\label{3.47}
\end{eqnarray}
By Cauchy Schwarz inequality we have for $|x_1|<N_m$, 
\begin{eqnarray}
\lefteqn{\int_{\mathbb{R}\setminus (-N_{m}^{+},N_{m}^{+})}\frac{|u^{(1)}(y_1,h)|}{(1+|x_1-y_1|)^{3/2}}dy_1}
\nonumber\\
&\leq&\biggl(\int_{\mathbb{R}\setminus (-N_{m}^{+},N_{m}^{+})}|u^{(1)}(y_1,h)|^{2}dy_1\biggr)^{1/2}\biggl(\int_{\mathbb{R}\setminus (-N_{m}^{+},N_{m}^{+})}\frac{dy_1}{(1+y_1-|x_1|)^{3}}\biggr)^{1/2}
\nonumber\\
&\leq& \frac{C}{1+N^{+}_m-|x_1|},\label{3.48}
\end{eqnarray}
and
\begin{eqnarray}
\int_{K_{N^{+}_m}}\frac{|u^{(1)}(y_1,h)|}{(1+|x_1-y_1|)^{3/2}}dy_1
&\leq&\biggl(\int_{K_{N_m}}|u^{(1)}(y_1,h)|^{2}dy_1\biggr)^{1/2}\biggl(\int_{K^{+}_{N_m}}\frac{dy_1}{(1+y_1-|x_1|)^{3}}\biggr)^{1/2}
\nonumber\\
&\leq& \frac{C}{N^{1/8}_{m}(1+N_m-|x_1|)}.\label{3.49}
\end{eqnarray}
Therefore, we obtain
\begin{eqnarray}
\lefteqn{\int_{-N_m}^{N_m}\bigl(W_{\infty}(x_1)-W_{N}(x_1)\bigr)^2dx_1}
\nonumber\\
&\leq&C\int_{-N_m}^{N_m}\frac{dx_1}{(1+N^{+}_m-|x_1|)^2}+\frac{C}{N_m^{1/4}}\int_{-N_m}^{N_m}\frac{dx_1}{(1+N_m-|x_1|)^2}
\nonumber\\
&\leq&\frac{C}{1+N^{1/2}_m}+\frac{C}{N^{1/4}_{m}}\ \leq \ \frac{C}{N^{1/4}_{m}}.\label{3.50}
\end{eqnarray}
Therefore, Collecting (\ref{3.43}), (\ref{3.47}), and (\ref{3.50}) we conclude that $|L(N_m)|\leq C\frac{\phi(N_m)^{3}}{N_m^{1/8}}$. Since $\phi(N)=N^s$, if we choose $s \in (0,1)$ such that $3s<\frac{1}{8}$, that is, $0<s<\frac{1}{24}$, the term $\frac{\phi(N_m)^{3}}{N_m^{1/8}}$ converges to zero as $m \to \infty$. Therefore, $\mathrm{limsup_{N\to \infty}}L(N)\geq0$, which yields Step 3.
\vspace{5mm}\\
By taking $\mathrm{limsup_{N \to \infty}}$ in (\ref{3.29}) we have that  
\begin{eqnarray}
0&\geq&\frac{k}{2\pi} \sum_{j \in J} \Biggl[ \sum_{d_{l,j}>0}|a_{l,j}|^2d_{l,j} -\sum_{d_{l,j}<0}|a_{l,j}|^2d_{l,j}\Biggr]
\nonumber\\
&+&\mathrm{limsup}_{N \to \infty}\Bigl(J_{+}(N)+J_{-}(N)+L(N)\Bigr).\label{3.51}
\end{eqnarray} 
By Steps 2 and 3 and choosing $0<s<\frac{1}{24}$ the right hand side is non-negative. Therefore, $a_{l,j}=0$ for all $l,j$, which yields $u^{(2)}=0$. Theorem 3.1 has been shown, and in next section we will show the uniqueness of $u^{(1)}$.
\end{proof}
\section{Uniqueness of $u^{(1)}$}
In Section 4, we will show the following lemma.
\begin{lem}
If $u \in H^{1}_{loc}(\mathbb{R}^2_+)$ satisfies
\begin{description}
\item[(i)] $u \in H^{1}(\mathbb{R}\times (0,R))$ for all $R>0$,
\item[(ii)] $\Delta u+k^2(1+q)nu=0 \ \mathrm{in} \ \mathbb{R}^2_{+}$,
\item[(iii)]$u$ vanishes for $x_2=0$,
\item[(iv)] There exists $\phi \in L^{\infty}(\Gamma_h)\cap H^{1/2}(\Gamma_h)$ with $u(x)=2\int_{\Gamma_h}\phi(y)\frac{\partial\Phi_k(x,y)}{\partial y_2} ds(y)$ for $x_2>h$,
\end{description}
then, $u \in H^{1}_{0}(\mathbb{R}^{2}_{+})$.
\end{lem}
By using Lemma 4.1, we have the uniqueness of solution in Theorem 1.2.
\begin{thm}
Let Assumptions 1.1 and 2.1 hold and let $k>0$ be regular in the sense of Definition 2.3. If $u \in H^{1}_{loc}(\mathbb{R}^2_{+})$ satisfies (\ref{3.1}), (\ref{3.2}), and the radiation condition in the sense of Definition 2.4, then $u$ vanishes for $x_2>0$.
\end{thm}
\begin{proof}[{\bf Proof of Theorem 4.2}]
Let $u \in H^{1}_{loc}(\mathbb{R}^2_{+})$ satisfy (\ref{3.1}), (\ref{3.2}), and the radiation condition in the sense of Definition 2.4. By Lemma 3.1, $u^{(2)} = 0$ for $x_2>0$. Then, $u^{(1)}$ satisfies the assumptions (i)--(iv) of Lemma 4.1, which implies that $u^{(1)} \in H^{1}_{0}(\mathbb{R}^{2}_{+})$. By Assumption 1.1, $u^{(1)}$ vanishes for $x_2>0$, which yields the uniqueness.
\end{proof}
\begin{proof}[{\bf Proof of Lemma 4.1}]
Let $R>h$ be fixed. We set $\Omega_{N,R}:=(-N,N) \times (0, R)$ where $N>0$ is large enough. We denote by $I^R_{\pm N}:=\{\pm N \}\times (0,R)$, $\Gamma_{R, N}:=(-N,N)\times \{R \}$, and $\Gamma_{R}:=(-\infty,\infty) \times \{R \}$. By Green's first theorem in $\Omega_{N,R}$ and assumptions {\bf(ii)},  {\bf(iii)} we have 
\begin{eqnarray}
\lefteqn{ \int_{\Omega_{N,R}}\{-k^2(1+q)n|u|^{2}+|\nabla u|^{2} \}dx=\int_{\Omega_{N,R}}\{ \overline{u}\Delta u+|\nabla u|^{2} \}dx}
\nonumber\\
&=&\int_{I^R_{N}} \overline{u}\frac{\partial u}{\partial x_1}  ds-\int_{I^R_{-N}} \overline{u}\frac{\partial u}{\partial x_1} ds +\int_{\Gamma_{R,N}} \overline{u}\frac{\partial u}{\partial x_2} ds. \label{4.3}
\end{eqnarray}
By the assumption {\bf (i)}, the first and second term in the right hands side of (\ref{4.3}) go to zero as $N \to \infty$. Then, by taking an imaginary part and as $N \to \infty$ in (\ref{4.3}) we have
\begin{equation}
\mathrm{Im} \int_{\Gamma_R} \overline{u} \frac{\partial u}{\partial x_2}ds = 0. \label{4.4}
\end{equation}
By considering the Floquet Bloch transform with respect to $x_1$ (see the notation of (\ref{2.5})), we can show that 
\begin{equation}
\int_{\Gamma_R} \overline{u} \frac{\partial u}{\partial x_2}ds = \int_{-1/2}^{1/2}\int_{0}^{2\pi} \overline{\tilde{u}_{\alpha}}(x_1, R)\frac{\partial \tilde{u}_{\alpha}(x_1, R)}{\partial x_2}dx_1d\alpha. \label{4.5}
\end{equation}
Since the upward propagating radiation condition is equivalent to the Rayleigh expansion by the Floquet Bloch transform (see the proof of Theorem 6.8 in \cite{Kirsch and Lechleiter2}), we can show that
\begin{equation}
\tilde{u}_{\alpha}(x)=\sum_{n \in \mathbb{Z}}u_{n}(\alpha)e^{inx_1+i\sqrt{k^2-(n+\alpha)^2}(x_2-h)}, \  x_2>h, \label{4.6}
\end{equation}
where $u_{n}(\alpha):=(2\pi)^{-1}\int_{0}^{2\pi}u_{\alpha}(x_1,h)e^{-inx_1}dx_1$. From (\ref{4.4})--(\ref{4.6}) we obtain that
\begin{eqnarray}
0&=&\mathrm{Im}\int_{-1/2}^{1/2}\int_{0}^{2\pi} \overline{\tilde{u}_{\alpha}}(x_1, R)\frac{\partial \tilde{u}_{\alpha}(x_1, R)}{\partial x_2}dx_1d\alpha
\nonumber\\
&=&\mathrm{Im}\sum_{n \in \mathbb{Z}}\int_{-1/2}^{1/2}2\pi|u_{n}(\alpha)|^2i\sqrt{k^2-(n+\alpha)^2}, \label{4.7}
\end{eqnarray}
Here, we denote by $k=n_0+r$ where $n_0 \in \mathbb{N}_0$ and $r \in [-1/2,1/2)$. Then by (\ref{4.7}) we have 
\begin{equation}
u_n(\alpha)=0 \ \mathrm{for} \ |n|<n_0, \ \mathrm{a.e.} \ \alpha \in (-1/2, 1/2), \nonumber
\end{equation}
\begin{equation}
u_{n_0}(\alpha)=0 \ \mathrm{for} \  \alpha \in (-1/2, r), \nonumber
\end{equation}
\begin{equation}
u_{-n_0}(\alpha)=0 \ \mathrm{for} \  \alpha \in (-r, 1/2) \label{4.8}.
\end{equation}
By (\ref{4.8}) we have
\begin{eqnarray}
&&\int_{-1/2}^{1/2}\int_{0}^{2\pi}\int_{R}^{\infty} |\tilde{u}_{\alpha}(x)|^2dx_2dx_1d\alpha
\nonumber\\
&=&2\pi \int_{-1/2}^{1/2} \sum_{|n|>n_0}|u_{n}(\alpha)|^2 \int_{R}^{\infty}e^{-\sqrt{(n+\alpha)^2-k^2}(x_2-h)}dx_2d\alpha \nonumber\\
&+&2\pi \int_{r}^{1/2} |u_{n_0}(\alpha)|^2 \int_{R}^{\infty}e^{-\sqrt{(n_0+\alpha)^2-k^2}(x_2-h)}dx_2d\alpha\nonumber\\
&+&2\pi \int_{-1/2}^{-r} |u_{-n_0}(\alpha)|^2 \int_{R}^{\infty}e^{-\sqrt{(-n_0+\alpha)^2-k^2}(x_2-h)}dx_2d\alpha \nonumber
\end{eqnarray}
\begin{eqnarray}
&\leq& 2\pi \sum_{|n|>n_0} \int_{-1/2}^{1/2} \frac{|u_{n}(\alpha)|^2 e^{-\sqrt{(n+\alpha)^2-k^2}(R-h)} }{\sqrt{(n+\alpha)^2-k^2}}d\alpha \nonumber\\
&+& 2\pi \int_{r}^{1/2} \frac{|u_{n_0}(\alpha)|^2 e^{-\sqrt{(n_0+\alpha)^2-k^2}(R-h)}}{\sqrt{(n_0+\alpha)^2-k^2}}d\alpha \nonumber\\
&+& 2\pi \int_{-1/2}^{-r} \frac{|u_{-n_0}(\alpha)|^2 e^{-\sqrt{(-n_0+\alpha)^2-k^2}(R-h)}}{\sqrt{(-n_0+\alpha)^2-k^2}}d\alpha \nonumber\\
&\leq& C \sum_{|n|>n_0} \int_{-1/2}^{1/2} |u_{n}(\alpha)|^2d\alpha \nonumber\\
&+& C \int_{r}^{1/2} \frac{|u_{n_0}(\alpha)|^2}{\sqrt{\alpha-r}}d\alpha + C \int_{-1/2}^{-r} \frac{|u_{-n_0}(\alpha)|^2}{\sqrt{-\alpha-r}}d\alpha,
\label{4.9}
\end{eqnarray}
and
\begin{eqnarray}
&&\int_{-1/2}^{1/2}\int_{0}^{2\pi}\int_{R}^{\infty} |\partial_{x_1} \tilde{u}_{\alpha}(x)|^2dx_2dx_1d\alpha \nonumber\\
&=&2\pi \sum_{|n|>n_0} \int_{-1/2}^{1/2} \frac{|u_{n}(\alpha)|^2 n^2 e^{-\sqrt{(n+\alpha)^2-k^2}(R-h)}}{\sqrt{(n+\alpha)^2-k^2}}d\alpha \nonumber\\
&+&2\pi \int_{r}^{1/2} \frac{|u_{n_0}(\alpha)|^2 n^2_0 e^{-\sqrt{(n_0+\alpha)^2-k^2}(R-h)}}{\sqrt{(n_0+\alpha)^2-k^2}}d\alpha \nonumber\\
&+&2\pi \int_{-1/2}^{-r} \frac{|u_{-n_0}(\alpha)|^2 n^2_0 e^{-\sqrt{(-n_0+\alpha)^2-k^2}(R-h)}}{\sqrt{(-n_0+\alpha)^2-k^2}}d\alpha \nonumber
\end{eqnarray}
\begin{eqnarray}
&\leq& C \sum_{|n|>n_0} \int_{-1/2}^{1/2} |u_{n}(\alpha)|^2d\alpha \nonumber\\
&+& C \int_{r}^{1/2} \frac{|u_{n_0}(\alpha)|^2}{\sqrt{\alpha-r}}d\alpha + C \int_{-1/2}^{-r} \frac{|u_{-n_0}(\alpha)|^2}{\sqrt{-\alpha-r}}d\alpha.
\label{4.10}
\end{eqnarray}
By the same argument in (\ref{4.10}) we have 
\begin{equation}
\int_{-1/2}^{1/2}\int_{0}^{2\pi}\int_{R}^{\infty} |\partial_{x_2} \tilde{u}_{\alpha}(x)|^2dx_2dx_1d\alpha 
\leq C \sum_{|n|>n_0} \int_{-1/2}^{1/2} |u_{n}(\alpha)|^2d\alpha \nonumber
\end{equation}
\begin{equation}
\vspace{3mm}+ C \int_{r}^{1/2} \frac{|u_{n_0}(\alpha)|^2}{\sqrt{\alpha-r}}d\alpha + C \int_{-1/2}^{-r} \frac{|u_{-n_0}(\alpha)|^2}{\sqrt{-\alpha-r}}d\alpha.
\label{4.11}
\end{equation}
It is well known that the Floquet Bloch Transform is an isomorphism between $H^1(\mathbb{R}^2_{+})$ and $L^2\bigl((-1/2,1/2)_{\alpha}; H^1((0,2\pi)\times \mathbb{R})_{x}\bigr)$ (e.g., see Theorem 4 in \cite{Lechleiter2}). Therefore, we obtain from (\ref{4.9})--(\ref{4.11})
\begin{eqnarray}
\left\| u \right\|^2_{H^1(\mathbb{R}\times (R, \infty))}&\leq& C \int_{-1/2}^{1/2}\int_{0}^{2\pi}\int_{R}^{\infty}|\tilde{u}_{\alpha}(x)|^2+|\partial_{x_1} \tilde{u}_{\alpha}(x)|^2+ |\partial_{x_2} \tilde{u}_{\alpha}(x)|^2dx_2dx_1d\alpha \nonumber\\
&\leq& C \sum_{|n|>n_0} \int_{-1/2}^{1/2} |u_{n}(\alpha)|^2d\alpha \nonumber\\
&+& C \int_{r}^{1/2} \frac{|u_{n_0}(\alpha)|^2}{\sqrt{\alpha-r}}d\alpha + C \int_{-1/2}^{-r} \frac{|u_{-n_0}(\alpha)|^2}{\sqrt{-\alpha-r}}d\alpha.
\nonumber\\
&\leq& C\int_{-1/2}^{1/2} \int_{0}^{2\pi} |\tilde{u}_{\alpha}(x_1,h)|^2dx_1d\alpha \nonumber\\
&+& C \int_{r}^{1/2} \frac{|u_{n_0}(\alpha)|^2}{\sqrt{\alpha-r}}d\alpha + C \int_{-1/2}^{-r} \frac{|u_{-n_0}(\alpha)|^2}{\sqrt{-\alpha-r}}d\alpha.
\label{4.12}
\end{eqnarray}
If we can show that
\begin{equation}
\exists \delta>0 \ \ \mathrm{and} \ \ \exists C>0 \ \ \mathrm{s.t.} \ \ |u_{\pm n_0}(\alpha)| \leq C \ \ \mathrm{for} \ \ \mathrm{all} \ \ \alpha \in (-\delta \pm r,\delta \pm r),\label{4.13}
\end{equation}
then the right hands side of (\ref{4.12}) is finite, which yield Lemma 4.1. 
\par Finally, we will show (\ref{4.13}). By the same argument in section 3 of \cite{Kirsch and Lechleiter2} we have
\begin{equation}
(I-K_{\alpha})\tilde{u}_{\alpha}=f_{\alpha} \ \mathrm{in} \ H^{1}_{0,per}(C_h),\label{4.14}
\end{equation}
where the operator $K_{\alpha}$ is defined by (\ref{2.12}) and $f_{\alpha}:=-(T_{per}k^2nqu)(\cdot, \alpha)$. Since the function $k^2nqu$ has a compact support, $\left\|f_{\alpha} \right\|^2_{H^1(C_h)}$ is bounded with respect to $\alpha$. By Assumption 2.1 and the operator $K_{\alpha}$ is compact, $(I-K_{\alpha})$ is invertible if $\alpha \in A_k$. Since $\pm r \in A_k$,  $(I-K_{\pm})$ is invertible. Since the  exceptional values are finitely many (see Lemma 2.2), $(I-K_{\alpha})$ is also invertible if $\alpha$ is close to $\pm r$. Therefore, there exists $\delta>0$ such that $(I-K_{\alpha})$ is invertible for all $\alpha \in (-\delta+r,\delta+r) \cup (-\delta-r,\delta-r)$. 
\par
The operator $(I-K_{\alpha})$ is of the form
\begin{equation}
(I-K_{\alpha})=(I-K_{\pm r}) \Bigl(I-(I-K_{\pm r})^{-1}[I-K_{\pm r}-(I-K_{\alpha})] \Bigr)=(I-K_{\pm r})(I-M_{\alpha}),\label{4.15}
\end{equation}
where $M_{\alpha}:=(I-K_{\pm r})^{-1}(K_{\alpha}-K_{\pm r})$. Next, we will estimate $(K_{\alpha}-K_{\pm r})$. By the definition of $K_{\alpha}$ we have for all $v, w \in H^{1}_{0,per}(C_h)$,
\begin{eqnarray}
\langle (K_{\alpha}-K_{\pm r})v, w \rangle_{*}&=&-\int_{C_h}\left[i(\alpha \mp r) \biggl(v \frac{\partial \overline{w}}{\partial x_1} -\overline{v}\frac{\partial \overline{w}}{\partial x_1}
\biggr)+(\alpha^2-r^2)v\overline{w}\right]dx 
\nonumber\\
&+& 2\pi i \sum_{|n|\neq n_0}v_n\overline{w_n} \bigl( \sqrt{k^2-(n+\alpha)^2}-\sqrt{k^2-(n\pm r)^2} \bigr)
\nonumber\\
&+& 2\pi i \sum_{|n|= n_0}v_n\overline{w_n} \bigl( \sqrt{k^2-(n+\alpha)^2}-\sqrt{k^2-(n\pm r)^2} \bigr).\nonumber\\ \label{4.16}
\end{eqnarray}
Since 
\begin{equation}
|\sqrt{k^2-(n+\alpha)^2}-\sqrt{k^2-(n\pm r)^2}|= \biggl|\frac{\pm 2nr+r^2-2n\alpha-\alpha^2}{\sqrt{k^2-(n+\alpha)^2}+\sqrt{k^2-(n\pm r)^2}} \biggr|
\nonumber
\end{equation}
\begin{eqnarray} 
\leq \left\{ \begin{array}{ll}
\frac{|n||\alpha \pm r|+|r^2-\alpha^2|}{\sqrt{|k^2-(n\pm r)^2|}} & \quad \mbox{for $|n| \neq n_0$}  \\
 \frac{|n||\alpha \pm r|+|r^2-\alpha^2|}{\sqrt{|r+\alpha||r-\alpha|}} & \quad \mbox{for $|n|=n_0$}, \\
\end{array} \right. \label{4.17}
\end{eqnarray}
we have for all $\alpha \in (-\delta+r,\delta+r) \cup (-\delta-r,\delta-r)$ 
\begin{eqnarray}
|\langle (K_{\alpha}-K_{\pm r})v, w \rangle_{*}|&\leq&C|\alpha\mp r| \left\|v \right\|_{H^1(C_h)}\left\|w \right\|_{H^1(C_h)}
\nonumber\\
&+&C \sum_{|n|\neq n_0}|v_n||w_n| \frac{|n||\alpha\mp r|}{\sqrt{|k^2-(n\pm r)^2|}}
\nonumber\\
&+&C \sum_{|n|= n_0}|v_n||w_n| n_0 \sqrt{|\alpha \mp r|}
\nonumber\\
&\leq&C \sqrt{|\alpha \mp r|}\left\|v \right\|_{H^1(C_h)}\left\|w \right\|_{H^1(C_h)}.
\label{4.18}
\end{eqnarray}
(we retake very small $\delta>0$ if needed.) This implies that there is a constant number $C>0$ which is independent of $\alpha$ such that $\left\| K_{\alpha}-K_{\pm r} \right\| \leq C\sqrt{|\alpha \mp r|}$. Therefore, by the property of Neumann series, there is a small $\delta>0$ such that for all $\alpha \in (-\delta+r,\delta+r) \cup (-\delta-r,\delta-r)$  
\begin{equation}
(I-M_{\alpha})^{-1}=\sum_{n=0}^{\infty}M_{\alpha}^{n}\ \ \mathrm{and} \ \ \left\|M_{\alpha} \right\| \leq 1/2.\label{4.19}
\end{equation}
By Cauchy-Schwarz, the boundedness of trace operator, and (\ref{4.19}) we have
\begin{eqnarray}
|u_{\pm n_0}(\alpha)|&\leq& \int_{0}^{2\pi}|\tilde{u}_{\alpha}(x_1,h)|dx_1\leq C\left\|\tilde{u}_{\alpha} \right\|_{H^1(C_h)}
\nonumber\\
&=&
C\left\|(I-M_{\alpha})^{-1}(I-K_{\pm r})^{-1} f_{\alpha} \right\|_{H^1(C_h)} \nonumber\\
&\leq&
C\left\| (I-M_{\alpha})^{-1} \right\| \left\| (I-K_{\pm r})^{-1} f_{\alpha} \right\| \nonumber\\
&\leq&
C\sum_{n=0}^{\infty}\left\|M_{\alpha} \right\|^n<C\sum_{n=0}^{\infty}(1/2)^j <\infty,
\label{4.20}
\end{eqnarray}
where constant number $C>0$ is independent of $\alpha$. Therefore, we have shown (\ref{4.13}). 
\end{proof}
\section{Existence}
In previous sections we discussed the uniqueness of Theorem 1.2. In Section 5, we will show the existence. Let Assumptions 1.1 and 2.1 hold and let $k>0$ be regular in the sense of Definition 2.3. Let $f \in L^{2}(\mathbb{R}^2_{+})$ such that $\mathrm{supp}f=Q$. We define the solution operator $S:L^{2}(Q)\to L^{2}(Q)$ by
$Sg:=v\bigl|_{Q}$ where $v$ satisfies the radiation condition and 
\begin{equation}
\Delta v+k^2nv=g, \ \mathrm{in} \ \mathbb{R}^2_{+}, \label{5.1}
\end{equation}
\begin{equation}
v=0 \ \mathrm{on} \ \Gamma_0. \label{5.2}
\end{equation}
Remark that by Theorem 2.6 we can define such a operator $S$, and $S$ is a compact operator since the restriction to $Q$ of the solution $v$ is in $H^{1}(Q)$. We define the multiplication operator $M:L^{2}(Q)\to L^{2}(Q)$ by $Mh:=k^{2}nqh$. We will show the following lemma.
\begin{lem}
$I_{L^{2}(Q)}+SM$ is invertible.
\end{lem}
\begin{proof}[{\bf Proof of Lemma 5.1}]
By the definition of operators $S$ and $M$
we have $SMg=v\bigl|_Q$ where $v$ is a radiating solution of  (\ref{5.1})--(\ref{5.2}) replacing $g$ by $k^{2}nqg$. If we assume that $(I_{L^{2}(Q)}+SM)g=0$, then $g=-v\bigl|_{Q}$, which implies that $v$ satisfies $\Delta v+k^2n(1+q)v=0$ in $\mathbb{R}^2_{+}$. By the uniqueness we have $v=0$ in $\mathbb{R}^2_{+}$, which implies that $I_{L^{2}(Q)}+SM$ is injective. Since the operator $SM$ is compact, by Fredholm theory we conclude that $I_{L^{2}(Q)}+SM$ is invertible.
\end{proof}
We define $u$ as the solution of
\begin{equation}
\Delta u+k^2nu=f-M(I_{L^{2}(Q)}+SM)^{-1}Sf, \ \mathrm{in} \ \mathbb{R}^2_{+}. \label{5.3}
\end{equation}
satisfying the radiation condition and $u=0$ on $\Gamma_0$. Since
\begin{eqnarray}
u\bigl|_{Q}
&=&S(f-M(I_{L^{2}(Q)}+SM)^{-1}Sf)
\nonumber\\
&=&(I_{L^{2}(Q)}+SM)(I_{L^{2}(Q)}+SM)^{-1}Sf-SM(I_{L^{2}(Q)}+SM)^{-1}Sf
\nonumber\\
&=&(I_{L^{2}(Q)}+SM)^{-1}Sf,\label{5.4}
\end{eqnarray}
we have that
\begin{equation}
\Delta u+k^2nu=f-k^{2}nqu, \ \mathrm{in} \ \mathbb{R}^2_{+}, \label{5.5}
\end{equation}
and $u$ is a radiating solution of (\ref{1.8})--(\ref{1.9}). Therefore, Theorem 1.2 has been shown. 
\section{Example of Assumption 1.1}
In Section 6, we will show the following lemma in order to give one of the example of Assumption 1.1.
\begin{lem}
Let $q$ and $n$ satisfy that $\partial_2 \bigl((1+q)n\bigr) \geq 0$ in $W$, and let $v \in H^{1}(\mathbb{R}^2_+)$ satisfy (\ref{1.6})--(\ref{1.7}). Then, $v$ vanishes for $x_2>0$.
\end{lem}
\begin{proof}[{\bf Proof of Lemma 6.1}]
Let $R>h$ be fixed. For $N>0$ we set $\Omega_{N,R}:=(-N,N) \times (0, R)$ and $I_{\pm N}^{R}:=\{\pm N \}\times (0,R)$ and $\Gamma_{R, N}:=(-N,N)\times \{R \}$. By Green's first theorem in $\Omega_{N,R}$ we have
\begin{eqnarray}
\lefteqn{ \int_{\Omega_{N,R}}\{-k^2(1+q)n|v|^{2}+|\nabla v|^{2} \}dx=\int_{\Omega_{N,R}}\{ \overline{v}\Delta v+|\nabla u|^{2} \}dx}
\nonumber\\
&=&\int_{I_{N}^{R}} \overline{v}\partial_1 v ds-\int_{I_{-N}^{R}} \overline{v}\partial_1 v ds +\int_{\Gamma_{R,N}} \overline{v}\partial_2 v ds.\label{6.1}
\end{eqnarray}
Since $v \in H^{1}(\mathbb{R}^2_+)$ the first and second term in the right hand side of (\ref{1.6}) go to zero as $N \to \infty$. Then, by taking an imaginary part in (\ref{6.1}) and as $N \to \infty$ we have 
\begin{equation}
\mathrm{Im}\int_{\Gamma_{R}}\overline{v}\partial_2 v ds=0.\label{6.2}
\end{equation} 
By the simple calculation, we have
\begin{eqnarray} 
\lefteqn{2\mathrm{Re}\bigl(\partial_2\overline{v}(\Delta v+k^2(1+q)nv)\bigl)}
\nonumber\\
&=&2\mathrm{Re}\bigl(\nabla\cdot(\partial_2\overline{v}\nabla v)\bigr)-\partial_2(|\nabla v|^2)+k^2(1+q)n\partial_2(|v|^2), \label{6.3}
\end{eqnarray}
which implies that
\begin{eqnarray} 
\lefteqn{0=2\mathrm{Re}\int_{\Omega_{N,R}}\partial_2 \overline{v}\bigl(\Delta v+k^2(1+q)nv)\bigl)dx=2\mathrm{Re}\int_{\Omega_{N,R}}\nabla\cdot(\partial_2\overline{v}\nabla v)dx}
\nonumber\\
&-&\int_{\Omega_{N,R}}\partial_2(|\nabla v|^2)dx+\int_{\Omega_{N,R}}k^2(1+q)n\partial_2(|v|^2)dx
\nonumber\\
&=&2\mathrm{Re}\biggl(-\int_{\Gamma_{0,N}}\partial_2\overline{v}\partial_2vds+\int_{I_N^R}\partial_2\overline{v}\partial_1vds-\int_{I_{-N}^R}\partial_2\overline{v}\partial_1vds+\int_{\Gamma_{R,N}}\partial_2\overline{v}\partial_2vds \biggr)
\nonumber\\
&-&\biggl(-\int_{\Gamma_{0,N}}|\nabla v|^2ds+\int_{\Gamma_{R,N}}|\nabla v|^2ds \biggr)
\nonumber\\
&-&\int_{\Gamma_{0,N}}k^2(1+q)n|v|^2ds+\int_{\Gamma_{R,N}}k^2(1+q)n|v|^2ds-\int_{\Omega_{N,R}}k^2\partial_2 \bigl((1+q)n \bigr)|v|^2dx
\nonumber\\
&=&-\int_{\Gamma_{0,N}}|\partial_2v|^2ds+\int_{\Gamma_{R,N}}\bigl(|\partial_2v|^2-|\partial_1v|^2+k^2|v|^2\bigr)ds
\nonumber\\
&-&\int_{\Omega_{N,R} \cap W}k^2\partial_2 \bigl((1+q)n \bigr)|v|^2dx+o(1).\label{6.4}
\end{eqnarray}
Since $\partial_2 \bigl((1+q)n\bigr) \geq 0$ in $W$, we have
\begin{equation}
\int_{\Gamma_{0,N}}|\partial_2v|^2ds \leq \int_{\Gamma_{R,N}}\bigl(|\partial_2v|^2-|\partial_1v|^2+k^2|v|^2\bigr)ds+o(1).\label{6.5}
\end{equation}
By taking limit as $N \to \infty$ we have 
\begin{equation} 
\int_{\Gamma_{R}}|\partial_2 v|^2ds\leq \int_{\Gamma_{R}}\bigl(|\partial_2v|^2-|\partial_1v|^2+k^2|v|^2\bigr)ds.\label{6.6}
\end{equation}
By Lemma 6.1 of \cite{Chandler and Zhang2} we have 
\begin{equation} 
\int_{\Gamma_{R}}\bigl(|\partial_2v|^2-|\partial_1v|^2+k^2|v|^2\bigr)ds \leq
2\mathrm{Im}\int_{\Gamma_{R}}\overline{v}\partial_2 vds.\label{6.7}
\end{equation}
From (\ref{6.2}), (\ref{6.6}), and (\ref{6.7}) we obtain that $\partial_2 v=0$ on $\Gamma_0$. We also have $v=0$ on $\Gamma_0$, which implies that by Holmgren's theorem and unique continuation principle we conclude that $v=0$ in $\mathbb{R}^{2}_{+}$.
\end{proof}
\section*{Acknowledgments}
The author thanks to Professor Andreas Kirsch, who supports him in the study of this waveguide problem, and gives him many comments to improve this paper.


Graduate School of Mathematics, Nagoya University, Furocho, Chikusa-ku, Nagoya, 464-8602, Japan \par
e-mail: takashi.furuya0101@gmail.com

\end{document}